\documentclass{CustomArticle}

\usepackage{CommonMath}
\usepackage{ProjectMath}
\usepackage{TheoremParts}
\usepackage{SimpleAlgorithm}
\usepackage{LabelNamespaces}

\usepackage{graphicx}
\usepackage{subcaption}
\usepackage{booktabs}
\usepackage{multirow}
\usepackage[
  style=trad-abbrv,
  citestyle=numeric-comp,
  sorting=nyt
]{biblatex}
\usepackage{import}
\usepackage{hyperref}
\usepackage[capitalize,compress]{cleveref}

\crefname{equation}{}{}

\PaperGeneralInfo{
  title={Gradient Methods for Stochastic Optimization\\in Relative Scale},
  date={May 28, 2023},
  version={0.3.0}
}
\AddPaperAuthor{
  name={Yurii Nesterov},
  affiliation={%
    Center for Operations Research and Econometrics (CORE),
    Catholic University of Louvain (UCL)
  },
  email={yurii.nesterov@uclouvain.be}
}
\AddPaperAuthor{
  name={Anton Rodomanov},
  affiliation={%
    Institute of Information and Communication Technologies,
    Electronics and Applied Mathematics (ICTEAM),
    Catholic University of Louvain (UCL)
  },
  email={anton.rodomanov@uclouvain.be}
}
\PaperAbstract{%
  We propose a new concept of a relatively inexact stochastic subgradient and
  present novel first-order methods that can use such objects to approximately
  solve convex optimization problems in relative scale.
  An important example where relatively inexact subgradients naturally arise
  is given by the Power or Lanczos algorithms for computing an approximate
  leading eigenvector of a symmetric positive semidefinite matrix.
  Using these algorithms as subroutines in our methods, we get new optimization
  schemes that can provably solve certain large-scale Semidefinite Programming
  problems with relative accuracy guarantees by using only matrix-vector
  products.%
}
\PaperKeywords{%
  convex optimization,
  optimization in relative scale,
  gradient methods,
  randomization,
  convergence guarantees,
  eigenvalues,
  singular values,
  power method,
  Lanczos algorithm%
}
\PaperThanks{%
  This paper has received funding from the European Research Council (ERC) under
  the European Union's Horizon 2020 research and innovation programme
  (grant agreement No 788368).
}

\begin{document}
  \PrintTitleAndAbstract

  \section{Introduction}

\subsection{Motivation}

Semidefinite Programming (SDP) is an important class of optimization problems.
The standard methods for solving SDP problems are Interior-Point Methods
(IPMs)~\cite{Nesterov.Nemirovskii-94-InteriorPoint}.
These methods are based on Newton steps and are very efficient for small- and
medium-size problems.
In many cases, IPMs are able to find an approximate solution with
a high accuracy in several dozen of iterations.
However, IPMs have a significant drawback: they cannot be used for large-scale
problems that often arise in modern applications and for which computing even
one Newton step becomes too expensive.

The only way to solve large-scale SDP problems is to use first-order methods
relying on matrix-vector products.
Compared to IPMs, these methods have much cheaper iterations and compute less
accurate solutions.
However, the accuracy is usually not a problem since, in the majority of
applications involving large-scale problems, there is no need for high accuracy.

In this paper, we develop new algorithms for solving optimization
problems in \emph{relative scale} with provable efficiency guarantees.
In contrast to already existing approaches, our methods can work
with inexact (possibly stochastic) information about the objective function.
\subsection{Notation and Generalities}
\label{sec:NotationAndGeneralities}

In what follows, we denote by $\VectorSpace{E}$ a finite-dimensional real
vector space, and by $\VectorSpace{E}\Dual$ its dual space, formed by all
linear functions on $\VectorSpace{E}$.
The value of function $s \in \VectorSpace{E}\Dual$ at point $x \in
\VectorSpace{E}$ is denoted by $\DualPairing{s}{x}$.

Given a self-adjoint positive semidefinite linear
operator~$\Map{B}{\VectorSpace{E}}{\VectorSpace{E}\Dual}$,
we can define the following Euclidean seminorm in~$\VectorSpace{E}$:
\begin{equation}
  \label{eq:EuclideanSeminorm}
  \RelativeNorm{x}{B} \DefinedEqual \DualPairing{B x}{x}^{1 / 2},
  \qquad
  x \in \VectorSpace{E}.
\end{equation}
An important subspace for the seminorm~$\RelativeNorm{\cdot}{B}$ is
the kernel~$\ker B$.
For any $x \in \VectorSpace{E}$, it holds that
$\RelativeNorm{x}{B} = 0$ iff $x \in \ker B$.
Hence, the seminorm~$\RelativeNorm{\cdot}{B}$ is a norm iff $B$ is nondegenerate.
More generally, for (any) complementary
subspace~$(\ker B)^c \subseteq \VectorSpace{E}$ to~$\ker B$,
the restriction of~$\RelativeNorm{\cdot}{B}$ onto~$(\ker B)^c$
is a norm.
Note that $\RelativeNorm{\cdot}{B}$ is constant along~$\ker B$:
\begin{equation}
  \label{eq:SeminormIsConstantAlongKernel}
  \RelativeNorm{x + h}{B}
  =
  \RelativeNorm{x}{B},
  \qquad
  \forall x \in \VectorSpace{E}, \
  \forall h \in \ker B.
\end{equation}

Each seminorm $\RelativeNorm{\cdot}{B}$ induces the following (generalized) dual
norm in~$\VectorSpace{E}\Dual$:
\begin{equation}
  \label{eq:DualNorm}
  \RelativeDualNorm{s}{B}
  \DefinedEqual
  \sup_{x \in \VectorSpace{E}}
  \SetBuilder{\DualPairing{s}{x}}{\RelativeNorm{x}{B} \leq 1},
  \qquad
  s \in \VectorSpace{E}\Dual.
\end{equation}
Strictly speaking, $\RelativeDualNorm{\cdot}{B}$ is not a norm
in~$\VectorSpace{E}\Dual$, as it can take infinite values at certain points:%
\footnote{
  Hereinafter,
  for a linear subspace~$\VectorSpace{L} \subseteq \VectorSpace{E}$,
  $
    \VectorSpace{L}\OrthogonalComplement
    \DefinedEqual
    \SetBuilder{
      s \in \VectorSpace{E}\Dual
    }{
      \DualPairing{s}{x} = 0, \ \forall x \in \VectorSpace{L}
    }
  $
  denotes the orthogonal complement of~$\VectorSpace{L}$ in~$\VectorSpace{E}$.
}
\begin{equation}
  \label{eq:FinitenessOfRelativeDualNorm}
  \RelativeDualNorm{s}{B} < +\infty
  \iff
  s \in (\ker B)\OrthogonalComplement,
  \qquad
  \forall s \in \VectorSpace{E}\Dual.
\end{equation}
Nevertheless, when restricted to $(\ker B)\OrthogonalComplement$,
$\RelativeDualNorm{\cdot}{B}$ is indeed a norm.
In the special case when $B$ is nondegenerate,
$(\ker B)\OrthogonalComplement = \VectorSpace{E}\Dual$
and $\RelativeDualNorm{\cdot}{B}$ becomes a norm in~$\VectorSpace{E}\Dual$
given by $\RelativeDualNorm{s}{B} = \DualPairing{s}{B^{-1} s}^{1 / 2}$
for all $s \in \VectorSpace{E}\Dual$.
The following identity is often useful:
\begin{equation}
  \label{eq:SquaredDualNorm}
  \frac{1}{2} \RelativeDualNorm{s}{B}^2
  =
  \sup_{x \in \VectorSpace{E}} \Bigl[
    \DualPairing{s}{x} - \frac{1}{2} \RelativeNorm{x}{B}^2
  \Bigr],
  \qquad
  \forall s \in \VectorSpace{E}\Dual.
\end{equation}

If $\VectorSpace{E} = \RealField^n$, the space of $n$-dimensional real column
vectors, then we often use the standard scalar product
\[
  \InnerProduct{x}{y}
  =
  x\Transpose y
  =
  \sum_{i=1}^n x\UpperIndex{i} y\UpperIndex{i},
  \quad
  x, y \in \RealField^n.
\]
For $x \in \RealField^n$, the standard Euclidean norm is defined as
$\Norm{x} = \InnerProduct{x}{x}^{1/2}$.

The standard Euclidean sphere in~$\RealField^n$ is denoted
by~$\UnitSphere{n - 1}$.
The cone of $n$-dimensional vectors with nonnegative coordinates is denoted
by~$\NonnegativeRay^n$.
If all the coordinates are strictly positive, we use the
notation~$\PositiveRay^n$.

The notation~$\RealMatrices{m}{n}$ is used for the space of real
$m \times n$ matrices equipped with the standard Frobenius inner product:
\[
  \InnerProduct{X}{Y}
  =
  \sum_{i=1}^n \sum_{j=1}^m X\UpperIndex{i, j} Y\UpperIndex{i, j},
  \quad
  X, Y \in \RealMatrices{m}{n}.
\]
For $X \in \RealMatrices{m}{n}$, its standard Frobenius norm is defined as
$\FrobeniusNorm{X} = \InnerProduct{X}{X}^{1/2}$.

For the space of symmetric $n \times n$ matrices, we use notation
$\SymmetricMatrices{n}$.
All eigenvalues of a matrix~$X \in \SymmetricMatrices{n}$ are real,
and we denote by $\MaxEigenValue(X)$ and $\MinEigenValue(X)$
the maximal and the minimal ones, respectively.
A matrix~$X \in \SymmetricMatrices{n}$ is positive semidefinite iff
$\MinEigenValue(X) \geq 0$ (notation $X \succeq 0$).
The convex cone of $n \times n$ real positive semidefinite matrices is denoted
by~$\SymmetricPsdMatrices{n}$.

For each matrix~$A \in \RealMatrices{m}{n}$, we can construct two ``squares''
of~$A$---the matrices $A A\Transpose \in \SymmetricPsdMatrices{m}$
and~$A\Transpose A \in \SymmetricPsdMatrices{n}$---that share
$r \DefinedEqual \min\Set{m, n}$ common nonnegative eigenvalues.
The square roots of these common eigenvalues---the numbers
$\sigma_1, \ldots, \sigma_r \in \NonnegativeRay$---are called
the \emph{singular values} of~$A$.
The largest of them is the \emph{maximal singular value} of~$A$
denoted by~$\MaxSingularValue(A)$.

Singular values are often used for defining matrix norms.
For a matrix~$A \in \RealMatrices{m}{n}$
and a real~$p \in \ClosedClosedInterval{0}{\infty}$,
the \emph{Schatten $p$-norm} of~$A$, denoted by $\SchattenNorm{A}{p}$,
is defined as the $\ell_p$-norm of the vector of singular values of~$A$.
An important example is the Schatten infinity-norm also known as the
\emph{spectral norm}:
\[
  \SchattenNorm{A}{\infty}
  \DefinedEqual
  \MaxSingularValue(A)
  =
  \max_{\Norm{x} = 1} \Norm{A x}.
\]
  \section{Optimization in Relative Scale}

\subsection{Gradient Method with Relatively Inexact Stochastic Oracle}
\label{sec:GradientMethodWithRelativelyInexactStochasticOracle}
\UsingNamespace{GradientMethodWithRelativelyInexactStochasticOracle}

Consider the following optimization problem:
\begin{equation}
  \LocalLabel{eq:Problem}
  f^* \DefinedEqual \min_{x \in Q} f(x),
\end{equation}
where $\Map{f}{\VectorSpace{E}}{\RealField}$ is a convex function
and $Q \subseteq \VectorSpace{E}$ is a nonempty convex set.
We assume that this problem is well-posed in the sense that it admits a
solution.

For measuring distances in the space~$\VectorSpace{E}$, we will use the
Euclidean seminorm~$\RelativeNorm{\cdot}{B}$, where
$\Map{B}{\VectorSpace{E}}{\VectorSpace{E}\Dual}$ is a fixed self-adjoint
positive semidefinite linear operator.

Our main assumptions on problem~\eqref{\LocalName{eq:Problem}} are as follows.
First, we assume that the objective function is consistent with the seminorm,
in the sense that there exists a point $\hat{x}_0 \in Q$
and a constant $\gamma_0 > 0$ such that
\begin{equation}
  \LocalLabel{eq:FunctionIsConsistent}
  f(x) \geq \gamma_0 \RelativeNorm{x - \hat{x}_0}{B}^2,
  \qquad
  \forall x \in Q.
\end{equation}
Second, we assume we have access to a stochastic gradient oracle for the
objective function, specified by a random
variable~$\xi \DistributedAs P_{\xi}$ taking values in a certain set~$S_{\xi}$
and a mapping~$\Map{g}{\VectorSpace{E} \times S_{\xi}}{\VectorSpace{E}\Dual}$.
We assume that the stochastic gradient oracle may be biased but, on average,
the corresponding bias is uniformly bounded in relative scale:
there exists $\delta \in \ClosedOpenInterval{0}{1}$ such that
\begin{equation}
  \LocalLabel{eq:ApproximateSubgradient}
  f(y)
  \geq
  (1 - \delta) f(x) + \DualPairing{\Expectation_{\xi} [g(x, \xi)]}{y - x},
  \qquad
  \forall x, y \in \VectorSpace{E}.
\end{equation}
We also assume that the magnitude of stochastic subgradients is relatively
bounded w.r.t.\ the function~$f$: for some $L > 0$,
\begin{equation}
  \LocalLabel{eq:StochasticSubgradientIsConsistent}
  \Expectation_{\xi} [\RelativeDualNorm{g(x, \xi)}{B}^2] \leq 2 L f(x),
  \qquad
  \forall x \in \VectorSpace{E}.
\end{equation}
The point~$x_0$ and the constants~$\gamma_0$, $\delta$ and~$L$ are supposed to
be known.
Finally, we need the following technical assumption to guarantee that
problem~\eqref{\LocalName{eq:Problem}}, as well as certain auxiliary subproblems
arising in the method, are well-posed.
\begin{assumption}
  \LocalLabel{as:ClosednessOfSet}
  The set $Q + \ker B$ is closed.
\end{assumption}
\Cref{\LocalName{as:ClosednessOfSet}} is satisfied, in particular, when $Q$ is
closed and $B$ is nondegenerate, or when $Q$ is an affine subspace
(and $B$ is arbitrary).

The main auxiliary operation in our method will be the following gradient step:
\begin{equation}
  \LocalLabel{eq:GradientStep}
  T_Q(\bar{x}, g)
  \DefinedEqual
  \argmin_{x \in Q} \Bigl\{
    \DualPairing{g}{x} + \frac{1}{2} \RelativeNorm{x - \bar{x}}{B}^2
  \Bigr\},
  \qquad
  \bar{x} \in \VectorSpace{E}, \
  g \in (\ker B)\OrthogonalComplement.
\end{equation}
Note that problem~\eqref{\LocalName{eq:GradientStep}} may have
multiple solutions (when~$B$ is degenerate);
in this case, we allow $T_Q(\bar{x}, g)$ to be chosen arbitrarily among them.
Nevertheless, a solution to~\eqref{\LocalName{eq:GradientStep}} always exists.

\begin{lemma}
  \LocalLabel{th:CharacterizationOfGradientStep}
  \UsingNestedNamespace{GradientMethodWithRelativelyInexactStochasticOracle}{CharacterizationOfGradientStep}
  Under \cref{GradientMethodWithRelativelyInexactStochasticOracle::as:ClosednessOfSet},
  the point $T \DefinedEqual T_Q(\bar{x}, g)$ is well-defined
  for any $\bar{x} \in \VectorSpace{E}$
  and any $g \in (\ker B)\OrthogonalComplement$,
  in the sense that problem~\eqref{\MasterName{eq:GradientStep}} has a solution.
  This point is characterized by the following equivalent optimality
  conditions:
  \begin{gather}
    \LocalLabel{eq:SubgradientOptimalityCondition}
    \DualPairing{g + B (T - \bar{x})}{x - T} \geq 0,
    \qquad
    \forall x \in Q,
    \\
    \LocalLabel{eq:FunctionalOptimalityCondition}
    \DualPairing{g}{x - T} + \frac{1}{2} \RelativeNorm{x - \bar{x}}{B}^2
    \geq
    \frac{1}{2} \RelativeNorm{T - \bar{x}}{B}^2
    +
    \frac{1}{2} \RelativeNorm{x - T}{B}^2,
    \qquad
    \forall x \in Q.
  \end{gather}
\end{lemma}

\begin{proof}
  \UsingNestedNamespace{GradientMethodWithRelativelyInexactStochasticOracle::CharacterizationOfGradientStep}{Proof}
  Let $\Map{\phi}{\VectorSpace{E}}{\RealField}$ be the function
  $
    \phi(x)
    \DefinedEqual
    \DualPairing{g}{x} + \frac{1}{2} \RelativeNorm{x - \bar{x}}{B}^2
  $.
  Clearly, $\phi$ is closed.
  Also, $\phi$ is constant along~$\ker B$
  (i.e., $\phi(x + h) = \phi(x)$ for all~$x \in \VectorSpace{E}$
  and all~$h \in \ker B$)
  since so is $\RelativeNorm{\cdot}{B}$ (see \cref{sec:NotationAndGeneralities})
  and since $g \in (\ker B)\OrthogonalComplement$.
  Further, it is not difficult to see that
  the restriction of~$\phi$ onto~$(\ker B)^c$
  (a complementary subspace to~$\ker B$)
  has bounded sublevel sets as $\RelativeNorm{\cdot}{B}$
  is a norm on~$(\ker B)^c$
  (rather than a seminorm, see \cref{sec:NotationAndGeneralities}).
  In particular, $\phi$ restricted to any subset of~$(\ker B)^c$
  also has bounded sublevel sets.
  Applying now \cref{th:ExistenceOfMinimizerOnSubspace}
  (taking into account \cref{GradientMethodWithRelativelyInexactStochasticOracle::as:ClosednessOfSet}),
  we conclude that $\phi$ has a minimizer on~$Q$,
  and thus the point~$T$ is well-defined.

  Inequality~\eqref{\MasterName{eq:SubgradientOptimalityCondition}}
  is the standard first-order optimality condition:
  a point~$T \in Q$ is a minimizer of a differentiable convex function~$\phi$
  on a convex set~$Q$ iff
  $\DualPairing{\Gradient \phi(T)}{x - T} \geq 0$ for all $x \in Q$.
  Inequality~\eqref{\MasterName{eq:FunctionalOptimalityCondition}} is equivalent
  to~\eqref{\MasterName{eq:SubgradientOptimalityCondition}} in view of the
  identity
  \[
    \DualPairing{B (T - \bar{x})}{T - x}
    =
    \frac{1}{2} \RelativeNorm{T - \bar{x}}{B}^2
    +
    \frac{1}{2} \RelativeNorm{x - T}{B}^2
    -
    \frac{1}{2} \RelativeNorm{x - \bar{x}}{B}^2
  \]
  which can be easily verified directly.
\end{proof}

Let us present our method for finding an approximate solution to
problem~\eqref{\LocalName{eq:Problem}} in relative scale.

\begin{SimpleAlgorithm}[
  title = {Gradient Method with Relatively Inexact Stochastic Oracle},
  label = \LocalName{alg:Algorithm},
  width = 0.75\textwidth
]
  \begin{AlgorithmGroup}[Input]
    Stochastic oracle~$g$, initial point~$x_0 \in Q$, constants~$L, \delta > 0$.
  \end{AlgorithmGroup}
  \AlgorithmGroupSeparator
  \begin{AlgorithmGroup}
    \UsingNestedNamespace{GradientMethodWithRelativelyInexactStochasticOracle}{Algorithm}
    \begin{AlgorithmSteps}
      \AlgorithmStep
        Set $v_0 \DefinedEqual x_0$, $C_0 \DefinedEqual 0$ ($\in \RealField$).
      \AlgorithmStep
        Iterate for $k \geq 0$:
        \begin{AlgorithmSteps}
          \AlgorithmStep
            \LocalLabel{step:ComputeStochasticSubgradient}
            Compute stochastic gradient $g_k \DefinedEqual g(v_k, \xi_k)$,
            where $\xi_k \DistributedAs P_{\xi}$.
          \AlgorithmStep
            \LocalLabel{step:ChooseStepSize}
            Choose step size $a_k \in \OpenOpenInterval{0}{(1 - \delta) / L}$
            in a deterministic way.
          \AlgorithmStep
            \LocalLabel{step:ComputeNewPoint}
            Compute coefficients
            $c_k \DefinedEqual a_k (1 - \delta - L a_k)$,
            $C_{k + 1} \DefinedEqual C_k + c_k$ and
            the new output point
            $x_{k + 1} \DefinedEqual (C_k x_k + c_k v_k) / C_{k + 1}$.
          \AlgorithmStep
            \LocalLabel{step:UpdateProxCenter}
            Update prox center $v_{k + 1} \DefinedEqual T_Q(v_k, a_k g_k)$.
        \end{AlgorithmSteps}
    \end{AlgorithmSteps}
  \end{AlgorithmGroup}
\end{SimpleAlgorithm}

\Cref{\LocalName{alg:Algorithm}} constructs a sequence of random
points~$(x_k)_{k = 1}^{\infty}$ each of which depends on the realization of
i.i.d.\ random variables~$\xi_1, \ldots, \xi_k$.
Recall from \cref{\LocalName{eq:GradientStep}} that, in order for
\cref{\LocalName{Algorithm::step:UpdateProxCenter}} in this method
to be well-defined, the stochastic subgradient~$g_k$ should
belong to the subspace~$(\ker B)\OrthogonalComplement$
at each iteration $k \geq 0$.
Let us show that this is indeed the case, and follows from
assumption~\eqref{\LocalName{eq:StochasticSubgradientIsConsistent}}.

\begin{lemma}
  \LocalLabel{th:AlgorithmIsWellDefined}
  In \cref{\LocalName{alg:Algorithm}}, we have
  $g_k \in (\ker B)\OrthogonalComplement$ (a.s.) for all $k \geq 0$.
  Thus, at each iteration~$k \geq 0$, the computation of~$v_{k + 1}$ is
  well-defined (a.s.).
\end{lemma}

\begin{proof}
  Let $k \geq 0$ be arbitrary.
  According to the definition of $g_k$ at
  \cref{GradientMethodWithRelativelyInexactStochasticOracle::Algorithm::step:ComputeStochasticSubgradient}
  and
  \cref{\LocalName{eq:StochasticSubgradientIsConsistent}},
  \[
    \Expectation_{\xi_k} \RelativeDualNorm{g_k}{B}^2
    =
    \Expectation_{\xi_k} \RelativeDualNorm{g(v_k, \xi_k)}{B}^2
    \leq
    2 L f(v_k)
    <
    +\infty.
  \]
  This means that $\RelativeDualNorm{g_k}{B} < +\infty$ (a.s.).
  Hence, in view of \cref{eq:FinitenessOfRelativeDualNorm},
  $g_k \in (\ker B)\OrthogonalComplement$ (a.s.).
  Thus, $a_k g_k \in (\ker B)\OrthogonalComplement$ (a.s.),
  and hence, by \cref{\LocalName{th:CharacterizationOfGradientStep}},
  $v_{k + 1}$ is well-defined (a.s.).
\end{proof}

Observe that, by definition, for each $k \geq 0$, the point~$x_{k + 1}$
is a convex combination of~$x_k$ and~$v_k$.
Since the set~$Q$ is convex, and $x_0 = v_0 \in Q$ and $v_k \in Q$
(by its definition at~\cref{\LocalName{Algorithm::step:UpdateProxCenter}})
for all $k \geq 0$, we therefore have $x_{k + 1} \in Q$ for all $k \geq 0$.
Thus, the points~$(x_k)_{k = 1}^\infty$ constructed by
\cref{\LocalName{alg:Algorithm}} are all feasible.

Let us now establish a general convergence guarantee for
\cref{\LocalName{alg:Algorithm}}, which is valid for any choice
of the coefficients~$(a_k)_{k = 0}^\infty$.

\begin{lemma}
  \LocalLabel{th:ConvergenceRate}
  \UsingNestedNamespace{GradientMethodWithRelativelyInexactStochasticOracle}{ConvergenceRate}
  In \cref{\MasterName{alg:Algorithm}},
  at any iteration $k \geq 1$, we have, for all $x \in Q$,
  \begin{equation}
    \LocalLabel{eq:PreliminaryGuarantee}
    \Bigl[ \, \sum_{i = 0}^{k - 1} a_i (1 - \delta - L a_i) \Bigr]
    \Expectation [f(x_k)]
    \leq
    \frac{1}{2} \RelativeNorm{x - x_0}{B}^2
    +
    \Bigl[ \, \sum_{i = 0}^{k - 1} a_i \Bigr] f(x).
  \end{equation}
  Furthermore, if $x_0 = \hat{x}_0$
  (where $\hat{x}_0 \in Q$ satisfies
  \cref{GradientMethodWithRelativelyInexactStochasticOracle::eq:FunctionIsConsistent}),
  then, at any iteration%
  \footnote{%
    The fact that $\delta_k < 1$ follows from our assumption
    at \cref{GradientMethodWithRelativelyInexactStochasticOracle::Algorithm::step:ChooseStepSize}.
  }
  $k \geq 1$,
  \begin{equation}
    \LocalLabel{eq:FinalGuarantee}
    (1 - \delta_k) \Expectation [f(x_k)] \leq f^*,
    \qquad \text{where} \quad
    \delta_k
    \DefinedEqual
    \frac{
      1 + 2 \gamma_0 \sum_{i = 0}^{k - 1} a_i (\delta + L a_i)
    }{
      1 + 2 \gamma_0 \sum_{i = 0}^{k - 1} a_i
    }
    \quad
    (< 1).
  \end{equation}
\end{lemma}

\begin{proof}
  \UsingNestedNamespace{GradientMethodWithRelativelyInexactStochasticOracle::ConvergenceRate}{Proof}

  \ProofPart

  Let $x \in Q$ and $k \geq 0$ be arbitrary.
  By the definition of $v_{k + 1}$
  at \cref{GradientMethodWithRelativelyInexactStochasticOracle::Algorithm::step:UpdateProxCenter}
  and \cref{GradientMethodWithRelativelyInexactStochasticOracle::th:CharacterizationOfGradientStep},
  we have
  \[
    a_k \DualPairing{g_k}{x - v_{k + 1}}
    +
    \frac{1}{2} \RelativeNorm{x - v_k}{B}^2
    \geq
    \frac{1}{2} \RelativeNorm{v_{k + 1} - v_k}{B}^2
    +
    \frac{1}{2} \RelativeNorm{x - v_{k + 1}}{B}^2.
  \]
  Rearranging and using \cref{eq:SquaredDualNorm}, we obtain
  \begin{align*}
    \hspace{2em}&\hspace{-2em}
    \frac{1}{2} \RelativeNorm{x - v_{k + 1}}{B}^2
    -
    \frac{1}{2} \RelativeNorm{x - v_k}{B}^2
    \leq
    a_k \DualPairing{g_k}{x - v_{k + 1}}
    -
    \frac{1}{2} \RelativeNorm{v_{k + 1} - v_k}{B}^2
    \\
    &=
    a_k \DualPairing{g_k}{x - v_k}
    +
    a_k \DualPairing{g_k}{v_{k + 1} - v_k}
    -
    \frac{1}{2} \RelativeNorm{v_{k + 1} - v_k}{B}^2
    \\
    &\leq
    a_k \DualPairing{g_k}{x - v_k}
    +
    \frac{1}{2} a_k^2 \RelativeDualNorm{g_k}{B}^2.
  \end{align*}
  Recall that
  $g_k = g(v_k, \xi_k)$
  (see \cref{GradientMethodWithRelativelyInexactStochasticOracle::Algorithm::step:ComputeStochasticSubgradient}).
  Further, by the construction of \cref{GradientMethodWithRelativelyInexactStochasticOracle::alg:Algorithm},
  $a_k$ is deterministic, while
  $v_k$ and $x_{k + 1}$ are independent of~$\xi_{k + 1}$.
  Therefore, passing to expectations w.r.t.\ $\xi_{k + 1}$ in the above display
  and using
  \cref{%
    GradientMethodWithRelativelyInexactStochasticOracle::eq:ApproximateSubgradient,%
    GradientMethodWithRelativelyInexactStochasticOracle::eq:StochasticSubgradientIsConsistent%
  },
  we get
  \begin{align*}
    \frac{1}{2} \Expectation_{\xi_k} [\RelativeNorm{x - v_{k + 1}}{B}^2]
    -
    \frac{1}{2} \RelativeNorm{x - v_k}{B}^2
    &\leq
    a_k [f(x) - (1 - \delta) f(v_k)] + L a_k^2 f(v_k)
    \\
    &=
    a_k f(x) - c_k f(v_k).
  \end{align*}
  (The final identity follows from the definition of~$c_k$ at
  \cref{GradientMethodWithRelativelyInexactStochasticOracle::Algorithm::step:ComputeNewPoint}.)
  Passing to full expectations and rearranging, we get
  \[
    \Expectation [c_k f(v_k)]
    +
    \frac{1}{2} \Expectation [\RelativeNorm{x - v_{k + 1}}{B}^2]
    \leq
    a_k f(x)
    +
    \frac{1}{2} \Expectation [\RelativeNorm{x - v_k}{B}^2].
  \]
  Note that this inequality is valid for any~$k \geq 0$.

  Summing up the above inequalities for all indices~$0 \leq k' \leq k - 1$,
  where $k \geq 1$ is arbitrary, dropping the
  term~$\Expectation [\RelativeNorm{x - v_k}{B}^2] \geq 0$
  and recalling that $v_0 = x_0$, we obtain
  \[
    \Expectation \Bigl[ \,\sum_{i = 0}^{k - 1} c_i f(v_i) \Bigr]
    \leq
    \Bigl[ \,\sum_{i = 0}^{k - 1} a_i \Bigr] f(x)
    +
    \frac{1}{2} \RelativeNorm{x - x_0}{B}^2.
  \]
  On the other hand, from the definitions at
  \cref{GradientMethodWithRelativelyInexactStochasticOracle::Algorithm::step:ComputeNewPoint},
  it follows that $C_k x_k = \sum_{i = 0}^{k - 1} c_i v_i$
  with
  \begin{equation}
    \LocalLabel{eq:ScalingCoefficient}
    C_k
    =
    \sum_{i = 0}^{k - 1} c_i
    =
    \sum_{i = 0}^{k - 1} a_i (1 - \delta - L a_i)
  \end{equation}
  being a deterministic coefficient
  (since each $a_i$ is assumed to be so).
  Hence, by the convexity of $f$, the left-hand side in the above display
  is $\geq C_k \Expectation [f(x_k)]$.
  Substituting further \cref{\LocalName{eq:ScalingCoefficient}},
  we obtain \cref{\MasterName{eq:PreliminaryGuarantee}}.

  \ProofPart

  Let us prove \cref{\MasterName{eq:FinalGuarantee}}.
  Let $k \geq 1$ and $x \in Q$ be arbitrary.
  Putting together
  \cref{%
    \MasterName{eq:PreliminaryGuarantee},%
    GradientMethodWithRelativelyInexactStochasticOracle::eq:FunctionIsConsistent%
  },
  we get
  \[
    C_k \Expectation [f(x_k)]
    \leq
    \Bigl( \frac{1}{2 \gamma_0} + \sum_{i = 0}^{k - 1} a_i \Bigr) f(x).
  \]
  Hence, according to \cref{\LocalName{eq:ScalingCoefficient}},
  \[
    \frac{C_k}{\frac{1}{2 \gamma_0} + \sum_{i = 0}^{k - 1} a_i}
    =
    \frac{
      \sum_{i = 0}^{k - 1} a_i (1 - \delta - L a_i)
    }{
      \frac{1}{2 \gamma_0} + \sum_{i = 0}^{k - 1} a_i
    }
    =
    1
    -
    \frac{
      1 + 2 \gamma_0 \sum_{i = 0}^{k - 1} a_i (\delta + L a_i)
    }{
      1 + 2 \gamma_0 \sum_{i = 0}^{k - 1} a_i
    }
    =
    1 - \delta_k
  \]
  Combining the above two displays, we get
  $(1 - \delta_k) \Expectation [f(x_k)] \leq f(x)$.
  This proves \cref{\MasterName{eq:FinalGuarantee}}
  in view of the definition of~$f^*$ in
  \cref{GradientMethodWithRelativelyInexactStochasticOracle::eq:Problem}
  and the fact that $x \in Q$ was arbitrary.
\end{proof}

Thus, after $k \geq 1$ iterations,
\cref{\LocalName{alg:Algorithm}} generates a point~$x_k \in Q$
which is, on average, a $\delta_k$-approximate solution to
\cref{\LocalName{eq:Problem}} in relative scale.
Let us show that, by appropriately choosing step sizes~$a_k$
in the method, we can make~$\delta_k$ sufficiently small
(for a sufficiently large~$k$).

First, observe that, for any $k \geq 0$, we have
\begin{equation}
  \LocalLabel{eq:AlternativeFormulaForAccuracy}
  \delta_k
  =
  \delta
  +
  \frac{
    1 - \delta + 2 \gamma_0 L \sum_{i = 0}^{k - 1} a_i^2
  }{
    1 + 2 \gamma_0 \sum_{i = 0}^{k - 1} a_i
  }
  \geq
  \delta.
\end{equation}
Therefore, we cannot hope for
\cref{\LocalName{alg:Algorithm}}
to produce an approximation solution whose relative accuracy will be better
than that of the oracle itself.
This is quite natural.
At the same time, we can easily ensure that
$\delta_k \to \delta$ as $k \to \infty$.
For this, it suffices to choose step sizes~$a_k$ in such a way that
\[
  \sum_{k = 0}^{\infty} a_k = \infty,
  \qquad
  \sum_{k = 0}^{\infty} a_k^2 < \infty,
\]
which is a standard recipe for subgradient methods (see, e.g., Section~3.2.3
in~\cite{Nesterov-18-LecturesConvex}).

Let us now derive an optimal choice of step sizes
for \cref{\LocalName{alg:Algorithm}}.
This is easier when we fix the total number of steps, say, $N \geq 1$.
From \cref{\LocalName{eq:AlternativeFormulaForAccuracy}},
it is not difficult to see that $\delta_N$ is a symmetric convex function
of~$(a_i)_{i = 0}^{N - 1}$.
Hence, its minimum is attained at $a_i = a_N^*$, $0 \leq i \leq N - 1$,
where~$a_N^*$ minimizes the ratio
\begin{equation}
  \LocalLabel{eq:RatioForOptimalStepSize}
  \delta_N(a)
  \DefinedEqual
  \frac{1 + 2 \gamma_0 N a (\delta + L a)}{1 + 2 \gamma_0 N a}
\end{equation}
over all $a \in \OpenOpenInterval{0}{\frac{1 - \delta}{L}}$.
Differentiating~$\delta_N(a)$ in~$a$ and setting the derivative to zero,
we come to the following equation for~$a_N^*$:
\begin{equation}
  \LocalLabel{eq:EquationForOptimalStepSize}
  (\delta + 2 L a_N^*) (1 + 2 \gamma_0 N a_N^*)
  =
  1 + 2 \gamma_0 N a_N^* (\delta + L a_N^*).
\end{equation}
This is the quadratic equation
$2 \gamma_0 N L (a_N^*)^2 + 2 L a_N^* = 1 - \delta$
with a unique positive solution
\begin{equation}
  \LocalLabel{eq:OptimalStepSize}
  a_N^*
  =
  \frac{\sqrt{2 \gamma_0 N L (1 - \delta) + L^2} - L}{2 \gamma_0 N L}
  =
  \frac{1 - \delta}{\sqrt{2 \gamma_0 N L (1 - \delta) + L^2} + L}
  \quad \Bigl( < \frac{1 - \delta}{L} \Bigr).
\end{equation}
Substituting this value into \cref{\LocalName{eq:RatioForOptimalStepSize}}
and taking into account \cref{\LocalName{eq:EquationForOptimalStepSize}},
we obtain
\[
  \delta_N(a_N^*) - \delta
  =
  2 L a_N^*
  =
  \frac{
    2 \sqrt{L} \, (1 - \delta)
  }{
    \sqrt{2 \gamma_0 N (1 - \delta) + L} + \sqrt{L}
  }
  \leq
  \sqrt{\frac{2 L}{\gamma_0 N}}.
\]
Thus, for the optimal choice of the step
size~\eqref{\LocalName{eq:OptimalStepSize}},
we have the $\BigO(1 / \sqrt{N})$ convergence rate to the level~$\delta$.
In particular, for generating a point~$x_N \in Q$ such that
\[
  (1 - 2 \delta) \Expectation [f(x_N)] \leq f^*,
\]
it suffices to make
\begin{equation}
  \LocalLabel{eq:ComplexityForOptimalStepSizes}
  N
  \geq
  N(\delta)
  \DefinedEqual
  \frac{2 L}{\gamma_0 \delta^2}
\end{equation}
iterations of \cref{\LocalName{alg:Algorithm}} with
step sizes~\eqref{\LocalName{eq:OptimalStepSize}}.

Instead of the optimal step sizes \cref{\LocalName{eq:OptimalStepSize}},
we can use another (simpler) choice that leads to the same complexity
guarantee~\cref{\LocalName{eq:ComplexityForOptimalStepSizes}}
but requires only the knowledge of~$\delta$ and~$L$.

\begin{theorem}
  \LocalLabel{th:ConvergenceRateForSimplerChoiceOfStepSizes}
  \UsingNestedNamespace{GradientMethodWithRelativelyInexactStochasticOracle}{ConvergenceRateForSimplerChoiceOfStepSizes}
  Consider \cref{\MasterName{alg:Algorithm}} with constant step
  sizes
  \begin{equation}
    \LocalLabel{eq:StepSizes}
    a_k \DefinedEqual \frac{\delta}{2 L} \quad \Bigl(< \frac{1}{L} \Bigr),
    \qquad
    k \geq 0,
  \end{equation}
  and the initial point $x_0 = \hat{x}_0$
  (where $\hat{x}_0 \in Q$ satisfies
  \cref{GradientMethodWithRelativelyInexactStochasticOracle::eq:FunctionIsConsistent}).
  Then, for any integer
  \begin{equation}
    \LocalLabel{eq:Complexity}
    N \geq N(\delta) \DefinedEqual \frac{2 L}{\gamma_0 \delta^2},
  \end{equation}
  we have
  \[
    (1 - 2 \delta) \Expectation [f(x_N)] \leq f^*.
  \]
\end{theorem}

\begin{proof}
  \UsingNestedNamespace{GradientMethodWithRelativelyInexactStochasticOracle::ConvergenceRateForSimplerChoiceOfStepSizes}{Proof}

  Substituting \cref{\MasterName{eq:StepSizes}} into
  \cref{GradientMethodWithRelativelyInexactStochasticOracle::eq:AlternativeFormulaForAccuracy},
  we obtain, for any $N \geq N(\delta)$,
  \begin{align*}
    \delta_N
    -
    \delta
    &=
    \frac{
      1 - \delta + 2 \gamma_0 L N \frac{\delta^2}{4 L^2}
    }{
      1 + 2 \gamma_0 N \frac{\delta}{2 L}
    }
    =
    \frac{2 L (1 - \delta) + \gamma_0 N \delta^2}{2 L + 2 \gamma_0 N \delta}
    \\
    &\leq
    \frac{L}{\gamma_0 N \delta} + \frac{\delta}{2}
    \leq
    \frac{L}{\gamma_0 N(\delta) \delta} + \frac{\delta}{2}
    =
    \delta,
  \end{align*}
  where the final identity follows from~\cref{\MasterName{eq:Complexity}}.
  It remains to apply \cref{GradientMethodWithRelativelyInexactStochasticOracle::th:ConvergenceRate}.
\end{proof}
\subsection{Dual Averaging Method}
\label{sec:DualAveragingMethod}

The Gradient Method from
\cref{sec:GradientMethodWithRelativelyInexactStochasticOracle}
has a couple of significant drawbacks.
First, it uses the same step size at every iteration, which is proportional to
the desired accuracy level
(formula~\eqref{GradientMethodWithRelativelyInexactStochasticOracle::ConvergenceRateForSimplerChoiceOfStepSizes::eq:StepSizes}).
This requires the user to know in advance the final accuracy they want to
obtain, and means that the method is essentially a \emph{short-step} one---its
step size is always small no matter what.
Instead, it would be more natural to start with a large step size at the initial
iterations, and then gradually decrease it.

Second, the algorithm works with an oracle whose accuracy is assumed to be
fixed.
However, in many cases, this quantity is actually a parameter of the oracle,
and one can query the oracle at the same point with various accuracies
(see \cref{sec:ApplicationsInSemidefiniteOptimization} for examples).
Since the complexity of an oracle is usually monotone in the required
accuracy, it makes sense to start with a large inaccuracy level and then
gradually decrease it in iterations, similarly to the above discussion on the
step sizes.

In this section, we present an algorithm that addresses the above drawbacks
and still enjoys the same worst-case complexity guarantee as the method
from \cref{sec:GradientMethodWithRelativelyInexactStochasticOracle}.
Our algorithm is based on the Dual Averaging method
from~\cite{Nesterov-07-PrimalDual}.

Our setup is almost the same as in
\cref{sec:GradientMethodWithRelativelyInexactStochasticOracle}.
Specifically, we are interested in solving the
problem~\eqref{GradientMethodWithRelativelyInexactStochasticOracle::eq:Problem}
under the
assumptions~\eqref{GradientMethodWithRelativelyInexactStochasticOracle::eq:FunctionIsConsistent},
\eqref{GradientMethodWithRelativelyInexactStochasticOracle::eq:ApproximateSubgradient}
and~\eqref{GradientMethodWithRelativelyInexactStochasticOracle::eq:StochasticSubgradientIsConsistent}.
The only difference is that now we assume that $\delta$ is also an input of the
oracle~$g$, so we should write $g(\delta, x, \xi)$ instead of $g(x, \xi)$.
(For simplicity, we assume that the constant~$L$ in
\cref{GradientMethodWithRelativelyInexactStochasticOracle::eq:StochasticSubgradientIsConsistent}
is independent of~$\delta$.)

We consider the following method:

\begin{SimpleAlgorithm}[
  title = {Dual Averaging with Relatively Inexact Stochastic Oracle},
  label = alg:DualAveraging,
  width = 0.65\linewidth
]
  \begin{AlgorithmGroup}[Input]
    Oracle~$g$, initial point~$x_0 \in Q$, constant~$L$,
    deterministic positive sequences~$(a_k)_{k = 1}^\infty$, $(\beta_k)_{k = 0}^\infty$
    and~$(\delta_k)_{k = 1}^\infty$.
  \end{AlgorithmGroup}
  \AlgorithmGroupSeparator
  \begin{AlgorithmGroup}[Preconditions]
    $(\beta_k)_{k = 0}^\infty$ is non-decreasing;
    $\frac{L a_k}{\beta_k} < 1 - \delta_k$ for all $k \geq 1$.
  \end{AlgorithmGroup}
  \AlgorithmGroupSeparator

  \begin{AlgorithmGroup}
    \begin{AlgorithmSteps}
      \AlgorithmStep
        \label{step:Initialization}
        $v_0 \DefinedEqual x_0$,
        $\bar{g}_0 \DefinedEqual 0$ ($\in \VectorSpace{E}\Dual$),
        $A_0 \DefinedEqual C_0 \DefinedEqual 0$ ($\in \RealField$).
      \AlgorithmStep
        Iterate for $k \geq 0$:
        \begin{AlgorithmSteps}
          \AlgorithmStep
            \label{step:ComputeShiftedProxCenter}
            $w_{k + 1} \DefinedEqual \bigl( \beta_k v_k + (\beta_{k + 1} - \beta_k) x_0 \bigr) / \beta_{k + 1}$.
          \AlgorithmStep
            $g_{k + 1} \DefinedEqual g(\delta_{k + 1}, w_{k + 1}, \xi_{k + 1})$
            for an i.i.d.\ sampled~$\xi_{k + 1}$.
          \AlgorithmStep
            \label{step:UpdateAverageGradient}
            $A_{k + 1} \DefinedEqual A_k + a_{k + 1}$,
            $\bar{g}_{k + 1} \DefinedEqual (A_k \bar{g}_k + a_{k + 1} g_{k + 1}) / A_{k + 1}$.
          \AlgorithmStep
            \label{step:UpdateProxCenter}
            $v_{k + 1} \DefinedEqual T_Q\bigl( x_0, \frac{A_{k + 1}}{\beta_{k + 1}} \bar{g}_{k + 1} \bigr)$.
          \AlgorithmStep
            \label{step:UpdateOutputPoint}
            $c_{k + 1} \DefinedEqual a_{k + 1} \bigl( 1 - \delta_{k + 1} - \frac{L a_{k + 1}}{\beta_{k + 1}} \bigr)$,
            $C_{k + 1} \DefinedEqual C_k + c_{k + 1}$,
            $x_{k + 1} \DefinedEqual (C_k x_k + c_{k + 1} w_{k + 1}) / C_{k + 1}$.
        \end{AlgorithmSteps}
    \end{AlgorithmSteps}
  \end{AlgorithmGroup}
\end{SimpleAlgorithm}

\begin{theorem}
  \label{th:DualAveragingConvergenceRate}
  Consider \cref{alg:DualAveraging} with $x_0 = \hat{x}_0$
  (see \cref{GradientMethodWithRelativelyInexactStochasticOracle::eq:FunctionIsConsistent}).
  Then, for any $k \geq 1$,
  \begin{equation}
    \label{eq:ConvergenceRate}
    (1 - \Delta_k) \Expectation f(x_k) \leq f^*,
    \qquad
    \text{where} \quad
    \Delta_k
    \DefinedEqual
    \frac{
      \beta_k + 2 \gamma_0 \sum_{i = 1}^k a_i (\delta_i + \frac{L a_i}{\beta_i})
    }{
      \beta_k + 2 \gamma_0 A_k
    }
    \quad (< 1).
  \end{equation}
\end{theorem}

\begin{proof}
  Let us define, for any $k \geq 0$, the following ``estimating function''
  $\Map{\psi_k}{\VectorSpace{E}}{\RealField}$:
  \begin{equation}
    \label{eq:EstimatingFunction}
    \psi_k(x)
    \DefinedEqual
    \frac{\beta_k}{2} \RelativeNorm{x - x_0}{B}^2
    +
    \sum_{i = 1}^k a_i \ell_i(x),
  \end{equation}
  where, for any $k \geq 1$,
  \[
    \ell_k(x) \DefinedEqual (1 - \delta_k) f(w_k) + \InnerProduct{g_k}{x - w_k}.
  \]
  Note that, according to
  \cref{GradientMethodWithRelativelyInexactStochasticOracle::eq:ApproximateSubgradient},
  in expectation, the function~$\ell_k$ is a lower bound on~$f$ over~$Q$:
  \[
    \Expectation_{\xi_k} \ell_k(x)
    =
    (1 - \delta_k) f(w_k) + \InnerProduct{\Expectation_{\xi_k} g_k}{x - w_k}
    \leq
    f(x),
    \qquad
    \forall x \in Q.
  \]
  Therefore, for any $k \geq 1$ and any $x \in Q$, we have
  \begin{equation}
    \label{eq:UpperBoundOnEstimatingFunction}
    \Expectation \psi_k(x)
    \leq
    \frac{\beta_k}{2} \RelativeNorm{x - x_0}{B}^2
    +
    A_k f(x).
  \end{equation}

  Let us show that, for any $k \geq 0$, we have
  \[
    v_k
    =
    \argmin_{x \in Q} \psi_k(x).
  \]
  This is obvious for $k = 0$ since, by our definition, $v_0 = x_0 \in Q$.
  Let $k \geq 1$.
  From the definitions at
  \cref{step:Initialization,step:UpdateAverageGradient},
  it follows that
  \[
    \bar{g}_k = \frac{1}{A_k} \sum_{i = 1}^k a_i g_i,
    \qquad
    A_k = \sum_{i = 1}^k a_i,
  \]
  Hence, according to its definition at \cref{step:UpdateProxCenter},
  for any $k \geq 1$, we have
  \[
    v_k
    =
    T_Q\Bigl(x_0, \frac{A_k}{\beta_k} \bar{g}_k \Bigr)
    =
    T_Q\Bigl(x_0, \frac{1}{\beta_k} \sum_{i = 1}^k a_i g_i \Bigr)
    =
    \argmin_{x \in Q} \psi_k(x),
  \]
  where the final identity is due to
  \cref{GradientMethodWithRelativelyInexactStochasticOracle::eq:GradientStep}.

  Since $v_k$ is the minimizer of~$\psi_k$ over~$Q$ and $\psi_k$ is a
  $1$-strongly convex function, we have, for any $k \geq 0$ and any $x \in Q$,
  \[
    \psi_k(x) \geq \psi_k^* + \frac{\beta_k}{2} \RelativeNorm{x - v_k}{B}^2,
  \]
  where $\psi_k^* \DefinedEqual \psi_k(v_k)$ be the minimal value of~$\psi_k$
  on~$Q$.

  Let $k \geq 0$ be an arbitrary index.
  According to \cref{eq:EstimatingFunction} and the above display,
  for any $x \in Q$, we have
  \begin{align*}
    \psi_{k + 1}(x)
    &=
    \psi_k(x)
    +
    \frac{\beta_{k + 1} - \beta_k}{2} \RelativeNorm{x - x_0}{B}^2
    +
    a_{k + 1} \ell_{k + 1}(x)
    \\
    &\geq
    \psi_k^*
    +
    \frac{\beta_k}{2} \RelativeNorm{x - v_k}{B}^2
    +
    \frac{\beta_{k + 1} - \beta_k}{2} \RelativeNorm{x - x_0}{B}^2
    +
    a_{k + 1} \ell_{k + 1}(x)
    \\
    &\geq
    \psi_k^*
    +
    \frac{\beta_{k + 1}}{2} \RelativeNorm{x - w_{k + 1}}{B}^2
    +
    a_{k + 1} \ell_{k + 1}(x),
  \end{align*}
  where the final inequality follows from the convexity of the squared
  (semi)norm and the definition of~$w_{k + 1}$ at
  \cref{step:ComputeShiftedProxCenter}
  (note that, according to our requirements, $\beta_k \leq \beta_{k + 1}$).
  Substituting now $x = v_{k + 1}$ together with the definition
  of~$\ell_{k + 1}(x)$, we get
  \begin{align*}
    \psi_{k + 1}^* - \psi_k^*
    &\geq
    \frac{\beta_{k + 1}}{2} \RelativeNorm{v_{k + 1} - w_{k + 1}}{B}^2
    +
    a_{k + 1} \bigl[
      (1 - \delta_{k + 1}) f(w_{k + 1})
      +
      \InnerProduct{g_{k + 1}}{v_{k + 1} - w_{k + 1}}
    \bigr]
    \\
    &\geq
    a_{k + 1} \Bigl[
      (1 - \delta_{k + 1}) f(w_{k + 1})
      -
      \frac{a_{k + 1}}{2 \beta_{k + 1}} \RelativeDualNorm{g_{k + 1}}{B}^2
    \Bigr].
  \end{align*}
  Taking the expectation w.r.t.~$\xi_{k + 1}$, using the fact that
  $
    \Expectation_{\xi_{k + 1}}[\RelativeDualNorm{g_{k + 1}}{B}^2]
    \leq
    2 L f(w_{k + 1})
  $
  (see \cref{GradientMethodWithRelativelyInexactStochasticOracle::eq:StochasticSubgradientIsConsistent}),
  and the definition of~$c_{k + 1}$ at \cref{step:UpdateOutputPoint},
  we obtain
  \[
    \Expectation_{\xi_{k + 1}} \psi_{k + 1}^* - \psi_k^*
    \geq
    a_{k + 1}
    \Bigl( 1 - \delta_{k + 1} - \frac{L a_{k + 1}}{\beta_{k + 1}} \Bigr)
    f(w_{k + 1})
    =
    c_{k + 1} f(w_{k + 1}).
  \]

  Passing to full expectations in the above inequalities, summing up
  and using the fact that $\psi_0^* = 0$, we get, for any $k \geq 1$,
  \begin{equation}
    \label{eq:LowerBoundForMinimalValueOfEstimatingFunction}
    \Expectation \psi_k^*
    \geq
    \sum_{i = 1}^k c_i \Expectation f(w_i)
    \geq
    C_k \Expectation f(x_k),
  \end{equation}
  where the final inequality follows from the convexity of~$f$
  and the fact that
  \[
    x_k = \frac{1}{C_k} \sum_{i = 1}^k c_i w_i,
    \qquad
    C_k = \sum_{i = 1}^k c_i
  \]
  (see the definitions at \cref{step:Initialization,step:UpdateOutputPoint}).

  Combining \cref{eq:LowerBoundForMinimalValueOfEstimatingFunction} with
  \cref{%
    eq:UpperBoundOnEstimatingFunction,%
    GradientMethodWithRelativelyInexactStochasticOracle::eq:FunctionIsConsistent%
  },
  we finally obtain, for any $x \in Q$,
  \[
    C_k \Expectation f(x_k)
    \leq
    \Expectation \psi_k^*
    \leq
    \Expectation \psi_k(x)
    \leq
    \frac{\beta_k}{2} \RelativeNorm{x - x_0}{B}^2 + A_k f(x)
    \leq
    \Bigl( \frac{\beta_k}{2 \gamma_0} + A_k \Bigr) f(x).
  \]
  This means that
  \[
    (1 - \Delta_k) \Expectation f(x_k) \leq f^*,
  \]
  where
  \begin{align*}
    \Delta_k
    &\DefinedEqual
    1 - \frac{C_k}{\frac{\beta_k}{2 \gamma_0} + A_k}
    =
    1
    -
    \frac{
      \sum_{i = 1}^k a_i (1 - \delta_i - \frac{L a_i}{\beta_i})
    }{
      \frac{\beta_k}{2 \gamma_0} + A_k
    }
    \\
    &=
    \frac{
      \frac{\beta_k}{2 \gamma_0}
      +
      \sum_{i = 1}^k a_i (\delta_i + \frac{L a_i}{\beta_i})
    }{
      \frac{\beta_k}{2 \gamma_0} + A_k
    }
    =
    \frac{
      \beta_k + 2 \gamma_0 \sum_{i = 1}^k a_i (\delta_i + \frac{L a_i}{\beta_i})
    }{
      \beta_k + 2 \gamma_0 A_k
    }.
  \end{align*}
  Note that $\Delta_k < 1$ in view of our assumption that
  $\frac{L a_k}{\beta_k} \leq 1 - \delta_k$ for all $i \geq 1$.
\end{proof}

Comparing the convergence rate estimate~\eqref{eq:ConvergenceRate}
with the corresponding
estimate~\eqref{GradientMethodWithRelativelyInexactStochasticOracle::ConvergenceRate::eq:FinalGuarantee}
for \cref{GradientMethodWithRelativelyInexactStochasticOracle::alg:Algorithm},
we see that they are very similar.
However, the former is more flexible.
First, it contains an additional control sequence~$\beta_k$ which can be chosen
by the user.
Second, the oracle inaccuracy~$\delta$ is allowed to vary at each iteration.
Note that, by choosing $\beta_k \equiv 1$ and $\delta_k \equiv \delta$
in \cref{alg:DualAveraging}, we obtain exactly the same estimate as for
\cref{GradientMethodWithRelativelyInexactStochasticOracle::alg:Algorithm}.

The particular form of the convergence rate estimate~\eqref{eq:ConvergenceRate}
suggests the following choice of oracle inaccuracies:
\begin{equation}
  \label{eq:ChoiceOfOracleInaccuracy}
  \delta_k \DefinedEqual \frac{L a_k}{\beta_k},
  \qquad
  k \geq 1.
\end{equation}
This choice is natural since then both terms~$\delta_i$
and~$\frac{L a_i}{\beta_i}$ in \cref{eq:ConvergenceRate} are well-balanced,
while the corresponding expression for~$\Delta_k$ is exactly the same,
up to an absolute constant, as if we had $\delta_i \equiv 0$:
\[
  \Delta_k
  =
  \frac{
    \beta_k + 4 \gamma_0 L \sum_{i = 1}^k \frac{a_i^2}{\beta_i}
  }{
    \beta_k + 2 \gamma_0 A_k
  }.
\]

Note that the above expression is very similar to that which arises
in the usual convergence analysis of the standard Dual Averaging method
(see~\cite{Nesterov-07-PrimalDual}).
In particular, dropping $\beta_k$ from the denominator in the above display,
we obtain
\[
  \Delta_k
  \leq
  \frac{\beta_k}{2 \gamma_0 A_k}
  +
  \frac{2 L}{A_k} \sum_{i = 1}^k \frac{a_i^2}{\beta_i}.
\]
Thus, we can try to use the same strategies for choosing the coefficients~$a_k$
and~$\beta_k$ as in the standard Dual Averaging method
from~\cite{Nesterov-07-PrimalDual}.
The simplest one is to choose
\begin{equation}
  \label{eq:ChoiceOfScalingCoefficients}
  a_k = 1, \qquad k \geq 1,
\end{equation}
and $\beta_k \sim \sqrt{\gamma_0 L k}$.
However, we still need to respect the constraint that
$\frac{L a_k}{\beta_k} < 1 - \delta_k$ for all $k \geq 1$
(see the preconditions in \cref{alg:DualAveraging}).
For our choice of~$a_k$ and $\delta_k$, this means $\beta_k > 2 L$
for all $k \geq 1$.
It is therefore reasonable to choose
\begin{equation}
  \label{eq:ChoiceOfProxCoefficients}
  \beta_k = \sqrt{8 \gamma_0 L k} + 2 L, \qquad k \geq 0.
\end{equation}
(The particular value of the absolute constant under the square root gives a
``resonably small'' abolute constant in the final iteration complexity bound below.)

\begin{theorem}
  \label{th:ComplexityEstimate}
  Consider \cref{alg:DualAveraging} with coefficients given by
  \cref{%
    eq:ChoiceOfOracleInaccuracy,%
    eq:ChoiceOfScalingCoefficients,%
    eq:ChoiceOfProxCoefficients%
  }.
  Then, for any $\delta \in \OpenOpenInterval{0}{1}$, we have
  \[
    (1 - \delta) \Expectation f(x_k) \leq f^*
  \]
  whenever $k \geq N(\delta)$, where
  \[
    N(\delta) \DefinedEqual \frac{10 L}{\gamma_0 \delta^2}.
  \]
\end{theorem}

\begin{proof}
  According to \cref{eq:ChoiceOfProxCoefficients}, $(\beta_k)_{k = 0}^\infty$ is
  an increasing sequence with $\beta_k > 2 L$ for all $k \geq 1$.
  Consequently, for any $k \geq 1$, we have
  $\frac{L a_k}{\beta_k} + \delta_k = \frac{2 L}{\beta_k} < 1$,
  so the preconditions of \cref{alg:DualAveraging} are satisfied.

  Let $\Delta_k$ be defined as in \cref{th:DualAveragingConvergenceRate}.
  We need to show that $\Delta_k \leq \delta$ for any $k \geq N(\delta)$.

  Let $k \geq 1$ be arbitrary.
  According to
  \cref{eq:ChoiceOfOracleInaccuracy,eq:ChoiceOfScalingCoefficients},
  we have $A_k = k$, and
  \[
    \Delta_k
    =
    \frac{
      \beta_k + 2 \gamma_0 \sum_{i = 1}^k a_i (\delta_i + \frac{L a_i}{\beta_i})
    }{
      \beta_k + 2 \gamma_0 A_k
    }
    =
    \frac{
      \beta_k + 4 \gamma_0 L \sum_{i = 1}^k \frac{1}{\beta_i}
    }{
      \beta_k + 2 \gamma_0 k
    }
    \leq
    \frac{\beta_k + 4 \gamma_0 L \sum_{i = 1}^k \frac{1}{\beta_i}}{2 \gamma_0 k}.
  \]
  Note that
  \[
    \sum_{i = 1}^k \frac{1}{\beta_i}
    =
    \sum_{i = 1}^k \frac{1}{\sqrt{8 \gamma_0 L i} + 2 L}
    \leq
    \frac{1}{\sqrt{8 \gamma_0 L}} \sum_{i = 1}^k \frac{1}{\sqrt{i}}
    \leq
    \sqrt{\frac{k}{2 \gamma_0 L}},
  \]
  where the final inequality follows from the fact that
  $
    \sum_{i = 1}^k \frac{1}{\sqrt{i}}
    \leq
    \int_0^k \frac{d t}{\sqrt{t}}
    =
    2 \sqrt{k}
  $.
  Thus,
  \[
    \Delta_k
    \leq
    \frac{
      \sqrt{8 \gamma_0 L k} + 2 L + 4 \gamma_0 L \sqrt{\frac{k}{2 \gamma_0 L}}
    }{
      2 \gamma_0 k
    }
    =
    \frac{2 \sqrt{8 \gamma_0 L k} + 2 L}{2 \gamma_0 k}
    =
    \sqrt{\frac{8 L}{\gamma_0 k}} + \frac{L}{\gamma_0 k}.
  \]

  Let $\delta \in \OpenOpenInterval{0}{1}$ be arbitrary.
  Denote $\tau_k \DefinedEqual \sqrt{\frac{L}{\gamma_0 k}}$.
  According to the above estimate, for any $k \geq 1$, we have
  $\Delta_k \leq 2 \sqrt{2} \tau_k + \tau_k^2$.
  To ensure that $\Delta_k \leq \delta$, it therefore suffices to ensure that
  $\tau_k \leq \tau$, where $\tau$ is the positive root of the following
  equation:
  \[
    2 \sqrt{2} \tau + \tau^2 = \delta.
  \]
  Solving this quadratic equation, we obtain
  \[
    \tau
    =
    \sqrt{2 + \delta} - \sqrt{2}
    =
    \frac{\delta}{\sqrt{2 + \delta} + \sqrt{2}}
    \geq
    \frac{\delta}{\sqrt{3} + \sqrt{2}}.
  \]
  Thus, it suffices to ensure that
  $\tau_k \leq \frac{\delta}{\sqrt{3} + \sqrt{2}}$,
  or, equivalently, that
  \[
    k
    \geq
    (\sqrt{3} + \sqrt{2})^2 \frac{L}{\gamma_0 \delta^2}
    =
    (5 + 2 \sqrt{6}) \frac{L}{\gamma_0 \delta^2}.
  \]
  It remains to note that $5 + 2 \sqrt{6} \leq 10$.
\end{proof}

Thus, \cref{alg:DualAveraging} with coefficients given by
\cref{%
  eq:ChoiceOfOracleInaccuracy,%
  eq:ChoiceOfScalingCoefficients,%
  eq:ChoiceOfProxCoefficients%
}
has the same worst-case iteration complexity (up to an absolute constant) as
\cref{GradientMethodWithRelativelyInexactStochasticOracle::alg:Algorithm}
with the fixed oracle inaccuracy~$\delta$.
Note, however, that, in our new method, the oracle inaccuracy~$\delta_k$
decreases at the following rate:
\[
  \delta_k
  =
  \frac{L}{\sqrt{8 \gamma_0 L k} + 2 L}
  =
  \frac{1}{\sqrt{\frac{8 \gamma_0 k}{L}} + 2}
  \sim
  \sqrt{\frac{L}{\gamma_0 k}}.
\]
In particular, for all $k \leq N(\delta) \sim \frac{L}{\gamma_0 \delta^2}$,
we have $\delta_k \gtrsim \delta$ which means that the new algorithm
never needs the oracle inaccuracy less than~$\delta$ (up to an absolute
constant).
For any reasonable oracle (whose complexity is polynomial in~$\delta$),
it means that the worst-case total complexity of the Dual Averaging method is
the same (in order) as that of the previously considered Gradient Method
with fixed step sizes.
If, however, the actual iteration complexity~$k(\delta)$ of the new
algorithm turns out to be significantly smaller than the worst-case~$N(\delta)$
on some particular problem, then the total oracle complexity of the new
algorithm may be significantly smaller as well.
\subsection{Composition with Affine Mapping}
\label{sec:CompositionWithAffineMapping}
\UsingNamespace{CompositionWithAffineMapping}

Let us show that our main assumptions from
\cref{sec:GradientMethodWithRelativelyInexactStochasticOracle}
are preserved when composing a function with an affine mapping.

Consider the problem
\begin{equation}
  \LocalLabel{eq:Problem}
  \min_{x \in Q} [f(x) \DefinedEqual F(A x + b)],
\end{equation}
where $Q \subseteq \VectorSpace{E}$ is a nonempty convex set,
$\Map{A}{\VectorSpace{E}}{\VectorSpace{E}_1}$ is a linear operator,
$b \in \VectorSpace{E}_1$,
and $\Map{F}{\VectorSpace{E}_1}{\RealField}$ is a convex function
satisfying assumptions from
\cref{sec:GradientMethodWithRelativelyInexactStochasticOracle}
on the set
\begin{equation}
  \LocalLabel{eq:SetImage}
  Q_1 \DefinedEqual A(Q) + b \quad (\subseteq \VectorSpace{E}_1).
\end{equation}
Specifically, we assume the space~$\VectorSpace{E}_1$ is equipped with
a certain Euclidean seminorm~$\RelativeNorm{\cdot}{B_1}$, where
$\Map{B_1}{\VectorSpace{E}_1}{\VectorSpace{E}_1\Dual}$
is a self-adjoint positive semidefinite linear operator,
and that the function~$F$ is consistent with this seminorm:
\begin{equation}
  \LocalLabel{eq:OuterFunctionIsConsistent}
  F(y) \geq \gamma_0 \RelativeNorm{y - \hat{y}_0}{B_1}^2,
  \qquad
  \forall y \in Q_1,
\end{equation}
where $\hat{y}_0 \in \VectorSpace{E}_1$ and $\gamma_0 > 0$.
We assume the following regularity condition is satisfied.
\begin{assumption}
  \LocalLabel{as:ClosednessOfSetImage}
  The set $Q_1 + \ker B_1$ is closed.
\end{assumption}

\Cref{\LocalName{as:ClosednessOfSetImage}} is satisfied, in particular,
when $Q$ is an affine subspace
(e.g., when problem~\eqref{\LocalName{eq:Problem}} is unconstrained,
i.e., $Q = \VectorSpace{E}$).

We also assume that the function~$F$ is represented by a relatively inexact
stochastic gradient oracle~$(G, \xi)$ with relatively bounded magnitude:
for some $\delta \in \ClosedOpenInterval{0}{1}$ and $L > 0$,
\begin{gather}
  \LocalLabel{eq:StochasticSubgradientForOuterFunctionIsRelativelyApproximate}
  F(y)
  \geq
  (1 - \delta) F(y_1) + \DualPairing{\Expectation_{\xi}[G(y_1, \xi)]}{y - y_1},
  \qquad
  \forall y, y_1 \in \VectorSpace{E}_1,
  \\
  \LocalLabel{eq:StochasticSubgradientForOuterFunctionIsRelativelyBounded}
  \Expectation_{\xi}[\RelativeDualNorm{G(y, \xi)}{B_1}^2]
  \leq
  2 L F(y),
  \qquad
  \forall y \in \VectorSpace{E}_1.
\end{gather}

A natural stochastic oracle for the function~$f$ is, of course,
$(g, \xi)$ defined by
\begin{equation}
  \LocalLabel{eq:StochasticSubgradient}
  g(x, \xi) \DefinedEqual A\Adjoint G(A x + b, \xi),
  \qquad
  x \in \VectorSpace{E}.
\end{equation}
It is not difficult to see that this oracle is $\delta$-relatively approximate
for~$f$: for any $x, y \in \VectorSpace{E}$, we have, according to our
assumption~\eqref{\LocalName{eq:StochasticSubgradientForOuterFunctionIsRelativelyApproximate}},
\begin{align*}
  f(y)
  &=
  F(A y + b)
  \geq
  (1 - \delta) F(A x + b)
  +
  \DualPairing{\Expectation_{\xi}[G(A x + b, \xi)]}{A (y - x)}
  \\
  &=
  (1 - \delta) f(x) + \DualPairing{\Expectation_{\xi}[g(x, \xi)]}{y - x}.
\end{align*}

Now let us introduce a Euclidean seminorm in the space~$\VectorSpace{E}$.
A good choice is
\begin{equation}
  \LocalLabel{eq:Seminorm}
  \RelativeNorm{x}{B}
  \DefinedEqual
  \RelativeNorm{A x}{B_1},
  \qquad
  x \in \VectorSpace{E},
\end{equation}
which is the seminorm induced by the operator
\begin{equation}
  \LocalLabel{eq:OperatorInducingSeminorm}
  B \DefinedEqual A\Adjoint B_1 A.
\end{equation}
This choice is good for several reasons.
First, it correctly ``translates'' our closedness
assumption~\ref{\LocalName{as:ClosednessOfSetImage}}
from the space~$\VectorSpace{E}_1$ into~$\VectorSpace{E}$.

\begin{lemma}
  Under \cref{\LocalName{as:ClosednessOfSetImage}},
  the set $Q + \ker B$ is closed.
\end{lemma}

\begin{proof}
  Note from \cref{\LocalName{eq:OperatorInducingSeminorm}} that
  $\ker B = \ker(B_1 A)$
  ($B_1$ is positive semidefinite).
  Combining this with \cref{\LocalName{eq:SetImage}}
  and the fact that closedness is a translation-invariant
  property, we see that we need to prove the following implication:
  \[
    A(Q) + \ker B_1 \ \text{is closed}
    \implies
    Q + \ker(B_1 A) \ \text{is closed}.
  \]
  But this follows from \cref{th:AuxiliaryClosednessResult}.
\end{proof}

Second, our choice of the seminorm preserves the consistency
constants~$\gamma_0$ and~$L$.
Indeed, let us define~$\hat{x}_0$ in the following way:
\begin{equation}
  \LocalLabel{eq:ReferencePointInConsistencyConditionForFunction}
  \hat{x}_0
  \DefinedEqual
  \argmin_{x \in Q} \RelativeNorm{A x + b - \hat{y}_0}{B_1}^2
  =
  T_Q\bigl( 0, A\Adjoint B_1 (b - \hat{y}_0) \bigr),
\end{equation}
where $T_Q$ is the gradient step defined in
\cref{GradientMethodWithRelativelyInexactStochasticOracle::eq:GradientStep}
(w.r.t.\ the seminorm~$\RelativeNorm{\cdot}{B}$ with $B$ given by
\cref{\LocalName{eq:OperatorInducingSeminorm}}).
The identity in
\cref{\LocalName{eq:ReferencePointInConsistencyConditionForFunction}}
follows from the fact that
\begin{align*}
  \RelativeNorm{A x + b - \hat{y}_0}{B_1}^2
  &=
  \RelativeNorm{A x}{B_1}^2
  +
  \DualPairing{B_1 A x}{b - \hat{y}_0}
  +
  \RelativeNorm{b - \hat{y}_0}{B_1}^2
  \\
  &=
  \RelativeNorm{x}{B}^2
  +
  \DualPairing{A\Adjoint B_1 (b - \hat{y}_0)}{x}
  +
  \RelativeNorm{b - \hat{y}_0}{B_1}^2.
\end{align*}
Observe that
$
  A\Adjoint B_1 (b - \hat{y}_0)
  \in
  ( \ker(B_1 A) )\OrthogonalComplement
  =
  (\ker B)\OrthogonalComplement,
$
hence, according to
\cref{GradientMethodWithRelativelyInexactStochasticOracle::th:CharacterizationOfGradientStep},
the point $\hat{x}_0$ is well-defined.

\begin{lemma}
  \LocalLabel{th:ConsistencyConditionsArePreserved}
  \UsingNestedNamespace{CompositionWithAffineMapping}{ConsistencyConditionsArePreserved}
  It holds that
  \begin{gather}
    \LocalLabel{eq:FunctionIsConsistent}
    f(x) \geq \gamma_0 \RelativeNorm{x - \hat{x}_0}{B}^2,
    \qquad
    \forall x \in Q.
    \\
    \LocalLabel{eq:StochasticSubgradientIsConsistent}
    \Expectation_{\xi} [\RelativeDualNorm{g(x, \xi)}{B}^2]
    \leq
    2 L f(x),
    \qquad
    \forall x \in \VectorSpace{E}.
  \end{gather}
\end{lemma}

\begin{proof}
  \UsingNestedNamespace{CompositionWithAffineMapping::ConsistencyConditionsArePreserved}{Proof}

  Let us prove \cref{\MasterName{eq:FunctionIsConsistent}}.
  Let $x \in Q$ be arbitrary.
  From
  \cref{%
    CompositionWithAffineMapping::eq:Problem,%
    CompositionWithAffineMapping::eq:OuterFunctionIsConsistent%
  },
  we get
  \[
    f(x)
    =
    F(A x + b)
    \geq
    \gamma_0 \RelativeNorm{A x + b - \hat{y}_0}{B_1}^2.
  \]
  It remains to prove that
  $
    \RelativeNorm{A x + b - \hat{y}_0}{B_1}
    \geq
    \RelativeNorm{x - \hat{x}_0}{B},
  $
  or, more generally, that
  \begin{equation}
    \LocalLabel{eq:InequalityToProve}
    \RelativeNorm{A x + b - \hat{y}_0}{B_1}^2
    \geq
    \RelativeNorm{A \hat{x}_0 + b - \hat{y}_0}{B_1}^2
    +
    \RelativeNorm{x - \hat{x}_0}{B}^2.
  \end{equation}
  This follows from
  \cref{CompositionWithAffineMapping::eq:ReferencePointInConsistencyConditionForFunction}.
  Indeed, by
  \cref{GradientMethodWithRelativelyInexactStochasticOracle::th:CharacterizationOfGradientStep},
  we have
  \[
    2 \DualPairing{A\Adjoint B_1 (b - \hat{y}_0)}{x - \hat{x}_0}
    +
    \RelativeNorm{x}{B}^2
    \geq
    \RelativeNorm{\hat{x}_0}{B}^2
    +
    \RelativeNorm{x - \hat{x}_0}{B}^2.
  \]
  Rearranging and using \cref{CompositionWithAffineMapping::eq:Seminorm},
  we can rewrite this as follows:
  \[
    2 \DualPairing{B_1 A x}{b - \hat{y}_0}
    +
    \RelativeNorm{A x}{B_1}^2
    \geq
    2 \DualPairing{B_1 A \hat{x}_0}{b - \hat{y}_0}
    +
    \RelativeNorm{A \hat{x}_0}{B_1}^2
    +
    \RelativeNorm{x - \hat{x}_0}{B}^2.
  \]
  Adding $\RelativeNorm{b - \hat{y}_0}{B_1}^2$ to both sides
  and completing the squares, we get \cref{\LocalName{eq:InequalityToProve}}.

  Let us prove \cref{\MasterName{eq:StochasticSubgradientIsConsistent}}.
  Let $x \in \VectorSpace{E}$ be arbitrary.
  According to
  \cref{%
    eq:DualNorm,%
    CompositionWithAffineMapping::eq:StochasticSubgradient,%
    CompositionWithAffineMapping::eq:Seminorm%
  },
  \begin{align*}
    \RelativeDualNorm{g(x, \xi)}{B}
    &=
    \sup_{h \in \VectorSpace{E}}
    \SetBuilder{\DualPairing{g(x, \xi)}{h}}{\RelativeNorm{h}{B} \leq 1}
    \\
    &=
    \sup_{h \in \VectorSpace{E}}
    \SetBuilder{
      \DualPairing{G(A x + b, \xi)}{A h}
    }{
      \RelativeNorm{A h}{B_1} \leq 1
    }
    \leq
    \RelativeDualNorm{G(A x + b, \xi)}{B_1}.
  \end{align*}
  Combining this with
  \cref{%
    CompositionWithAffineMapping::eq:StochasticSubgradientForOuterFunctionIsRelativelyBounded,%
    CompositionWithAffineMapping::eq:Problem%
  },
  we obtain
  \[
    \Expectation_{\xi} [\RelativeDualNorm{g(x, \xi)}{B}^2]
    \leq
    \Expectation_{\xi} [\RelativeDualNorm{G(A x + b, \xi)}{B_1}^2]
    \leq
    2 L F(A x + b)
    =
    2 L f(x).
    \qedhere
  \]
\end{proof}
  \section{Applications in Semidefinite Optimization}
\label{sec:ApplicationsInSemidefiniteOptimization}

Let us present several examples of relatively inexact stochastic subgradient
oracles, which are suitable for various functions of matrix eigen- and
singular values.

In this section, our random vectors belong to $\RealField^n$.
Therefore, we will use notation $\Norm{\cdot}$ for the standard Euclidean norm.

\subsection{Approximating Eigenvectors and Singular Vectors}

\subsubsection{Relatively Inexact Eigen- And Singular Vectors}

\begin{definition}[Relatively inexact maximal eigenvector]
  \label{def:RelativelyInexactMaximalEigenvector}
  Given a matrix~$A \in \SymmetricPsdMatrices{n}$,
  a scalar~$\delta \in \OpenOpenInterval{0}{1}$
  and a random vector~$\hat{v} \in \RealField^n$,
  we say that $\hat{v}$ is a \emph{$\delta$-relatively inexact stochastic
  maximal unit eigenvector} of~$A$
  iff $\Norm{\hat{v}} = 1$ (a.s.) and
  \[
    \Expectation \InnerProduct{A \hat{v}}{\hat{v}}
    \geq
    (1 - \delta) \MaxEigenValue(A).
  \]
  A \emph{relatively inexact stochastic maximal eigenvector oracle} is
  a procedure $\MaxEigenVector$ that takes
  a matrix $A \in \SymmetricPsdMatrices{n}$
  and a scalar $\delta \in \OpenOpenInterval{0}{1}$, and returns
  a random vector $\hat{v} = \MaxEigenVector(A, \delta)$
  such that $\hat{v}$ is a $\delta$-relatively inexact stochastic maximal
  unit eigenvector of~$A$.
\end{definition}

Sometimes, it is necessary to have some guarantees not only for the Rayleigh
quotient $\InnerProduct{A \hat{v}}{\hat{v}}$ but also for its powers.
Since the expectation is not, in general, invariant w.r.t.\ taking powers,
it makes sense to introduce the following generalization of
\cref{def:RelativelyInexactMaximalEigenvector}.

\begin{definition}
  Let $A \in \SymmetricPsdMatrices{n}$ be a matrix,
  and let $\delta \in \OpenOpenInterval{0}{1}$
  and $p \in \OpenOpenInterval{0}{\infty}$ be scalars.
  A random vector~$\hat{v} \in \RealField^n$
  is called a \emph{$\delta$-relatively inexact stochastic
  maximal unit eigenvector of~$A$ of degree~$p$}
  iff $\Norm{\hat{v}} = 1$ (a.s.) and
  \begin{equation}
    \label{eq:RelativelyInexactMaxEigenVectorOfDegree}
    \Expectation[\InnerProduct{A \hat{v}}{\hat{v}}^p]
    \geq
    (1 - \delta) [\MaxEigenValue(A)]^p.
  \end{equation}
  A \emph{relatively inexact stochastic maximal eigenvector oracle of
  degree~$p \in \OpenOpenInterval{0}{\infty}$}
  is a procedure $\MaxEigenVector_p$ that
  accepts a matrix~$A \in \SymmetricPsdMatrices{n}$ and a scalar
  $\delta \in \OpenOpenInterval{0}{1}$,
  and returns a random vector~$\hat{v} = \MaxEigenVector_p(A, \delta)$
  such that $\hat{v}$ is a $\delta$-relatively inexact stochastic maximal unit
  eigenvector of~$A$ of degree~$p$.
\end{definition}

Note that an oracle for $\MaxEigenVector_p(A, \delta)$
can be easily implemented in terms of that for
$\MaxEigenVector(A, \delta)$.
In what follows, we refer to this implementation as the default one.

\begin{SimpleAlgorithm}[
  label = alg:MakeMaxEigenVectorOfDegree,
  width = 0.5\linewidth
]
  \begin{AlgorithmGroup}[Signature]
    $\MaxEigenVector_p = \MakeMaxEigenVectorOfDegree(p, \MaxEigenVector)$.
  \end{AlgorithmGroup}
  \AlgorithmGroupSeparator

  \begin{AlgorithmGroup}[Input]
    $p$---degree [$p \in \OpenOpenInterval{0}{\infty}$];
    $\MaxEigenVector$---relatively inexact stochastic maximal eigenvector oracle.
  \end{AlgorithmGroup}
\begin{AlgorithmGroup}[Output]
    $\MaxEigenVector_p$---relatively inexact stochastic maximal eigenvector
    oracle of degree~$p$.
  \end{AlgorithmGroup}
\AlgorithmGroupSeparator

  \begin{AlgorithmGroup}
    \begin{AlgorithmSteps}
      \AlgorithmStep
        If $p \leq 1$, return $\MaxEigenVector$.
      \AlgorithmStep
        Otherwise, return the function defined by
        \[
          (A, \delta)
          \mapsto
          \MaxEigenVector\bigl( A, 1 - (1 - \delta)^{1 / p} \bigr).
        \]
    \end{AlgorithmSteps}
  \end{AlgorithmGroup}
\end{SimpleAlgorithm}

\begin{lemma}
  The output in \cref{alg:MakeMaxEigenVectorOfDegree} is indeed a relatively
  inexact stochastic eigenvector oracle of degree~$p$.
\end{lemma}

\begin{proof}
  Let $A \in \SymmetricPsdMatrices{n}$
  and $\delta \in \OpenOpenInterval{0}{1}$ be arbitrary,
  and let $\hat{v} \DefinedEqual \MaxEigenVector_p(A, \delta)$
  be the output of the $\MaxEigenVector_p$ procedure constructed by
  \cref{alg:MakeMaxEigenVectorOfDegree}.
  Let us show that $\hat{v}$ is a $\delta$-relatively inexact stochastic
  eigenvector of~$A$ of degree~$p$.

  Clearly, $\Norm{\hat{v}} = 1$ (a.s.) since $\hat{v}$ is the output
  of the $\MaxEigenVector$ procedure.
  Therefore, we only need to justify
  inequality~\eqref{eq:RelativelyInexactMaxEigenVectorOfDegree}.
  We may also assume that $\MaxEigenValue(A) > 0$ since otherwise the inequality
  is trivial in view of positive semidefiniteness of~$A$.

  If $p \leq 1$, then, by construction, $\hat{v} = \MaxEigenVector(A, \delta)$,
  and hence
  \[
    \Expectation \InnerProduct{A \hat{v}}{\hat{v}}
    \geq
    (1 - \delta) \MaxEigenValue(A).
  \]
  Since $\InnerProduct{A \hat{v}}{\hat{v}} \leq \MaxEigenValue(A)$
  and the function $p \mapsto \tau^p$ is monotonically decreasing
  on~$\OpenClosedInterval{0}{1}$ for any fixed
  $\tau \in \ClosedClosedInterval{0}{1}$, it follows that
  \[
    \Expectation\biggl[
      \biggl(
        \frac{\InnerProduct{A \hat{v}}{\hat{v}}}{\MaxEigenValue(A)}
      \biggr)^p
    \biggr]
    \geq
    \Expectation \frac{\InnerProduct{A \hat{v}}{\hat{v}}}{\MaxEigenValue(A)}
    \geq
    1 - \delta.
  \]
  Rearranging, we obtain
  inequality~\eqref{eq:RelativelyInexactMaxEigenVectorOfDegree}.

  If $p > 1$, then $\hat{v} = \MaxEigenVector(A, \delta_p)$
  for $\delta_p \DefinedEqual 1 - (1 - \delta)^{1 / p}$, which means that
  \[
    \Expectation \InnerProduct{A \hat{v}}{\hat{v}}
    \geq
    (1 - \delta_p) \MaxEigenValue(A)
    =
    (1 - \delta)^{1 / p} \MaxEigenValue(A).
  \]
  Applying now Jensen's inequality (using the fact that $t \mapsto t^p$ is a
  convex and monotonically increasing function on~$\NonnegativeRay$), we obtain
  \[
    \Expectation[\InnerProduct{A \hat{v}}{\hat{v}}^p]
    \geq
    [\Expectation \InnerProduct{A \hat{v}}{\hat{v}}]^p
    \geq
    \bigl[ (1 - \delta)^{1 / p} \MaxEigenValue(A) \bigr]^p
    =
    (1 - \delta) [\MaxEigenValue(A)]^p.
    \qedhere
  \]
\end{proof}

\begin{definition}[Relatively inexact maximal left/right singular vectors]
  Let $A \in \RealMatrices{m}{n}$ be a matrix
  and $\delta \in \OpenOpenInterval{0}{1}$ be a scalar.
  A random vector~$\hat{u} \in \RealField^m$ is called a
  \emph{$\delta$-relatively inexact stochastic maximal left singular vector}
  of~$A$ iff $\Norm{\hat{u}} = 1$ (a.s.) and
  \[
    \Expectation \Norm{A\Transpose \hat{u}}
    \geq
    (1 - \delta) \MaxSingularValue(A).
  \]
  A \emph{relatively inexact stochastic maximal left singular vector oracle}
  is a procedure $\MaxLeftSingularVector$ that takes
  a matrix~$A \in \RealMatrices{n}{m}$ and
  a scalar~$\delta \in \OpenOpenInterval{0}{1}$,
  and returns a random vector~$\hat{u} = \MaxLeftSingularVector(A, \delta)$
  such that $\hat{u}$ is a $\delta$-relatively inexact stochastic maximal left
  singular vector.

  Similarly, a random vector~$\hat{v} \in \RealField^n$ is called
  a \emph{$\delta$-relatively inexact maximal right singular vector} of~$A$
  iff $\Norm{\hat{v}} = 1$ (a.s.) and
  \[
    \Expectation \Norm{A \hat{v}}
    \geq
    (1 - \delta) \MaxSingularValue(A).
  \]
\end{definition}

The maximal left/right singular vector oracles can be implemented in terms of
the maximal eigenvector oracle of degree~$1 / 2$.

\begin{center}
  \begin{minipage}[t]{0.4\linewidth}
    \begin{SimpleAlgorithm}[%
      label = alg:MakeMaxLeftSingularVector,
      width = \linewidth
    ]
      \begin{AlgorithmGroup}[Signature]
        $\MaxLeftSingularVector = \MakeMaxLeftSingularVector(\MaxEigenVector_{1 / 2})$
      \end{AlgorithmGroup}
      \AlgorithmGroupSeparator

      \begin{AlgorithmGroup}[Input]
        $\MaxEigenVector_{1 / 2}$---relatively inexact maximal eigenvector
        oracle of degree~$1 / 2$.
      \end{AlgorithmGroup}
      \begin{AlgorithmGroup}[Output]
        $\MaxLeftSingularVector$---relatively inexact maximal left singular
        vector oracle.
      \end{AlgorithmGroup}
      \AlgorithmGroupSeparator

      \begin{AlgorithmGroup}
        Return the function
        $
          (A, \delta)
          \mapsto
          \MaxEigenVector_{1 / 2}(A A\Transpose, \delta).
        $
      \end{AlgorithmGroup}
    \end{SimpleAlgorithm}
  \end{minipage}
  \hspace{2em}
  \begin{minipage}[t]{0.45\linewidth}
      \begin{SimpleAlgorithm}[%
        label = alg:MakeMaxRightSingularVector,
        width = \linewidth
      ]
      \begin{AlgorithmGroup}[Signature]
        $\MaxRightSingularVector = \MakeMaxRightSingularVector(\MaxEigenVector_{1 / 2})$
      \end{AlgorithmGroup}
      \AlgorithmGroupSeparator

      \begin{AlgorithmGroup}[Input]
        $\MaxEigenVector_{1 / 2}$---relatively inexact maximal eigenvector
        oracle of degree~$1 / 2$.
      \end{AlgorithmGroup}
      \begin{AlgorithmGroup}[Output]
        $\MaxRightSingularVector$---relatively inexact maximal left singular
        vector oracle.
      \end{AlgorithmGroup}
      \AlgorithmGroupSeparator

      \begin{AlgorithmGroup}
        Return the function
        $
          (A, \delta)
          \mapsto
          \MaxEigenVector_{1 / 2}(A\Transpose A, \delta).
        $
      \end{AlgorithmGroup}
    \end{SimpleAlgorithm}
  \end{minipage}
\end{center}

\begin{lemma}
  \Cref{alg:MakeMaxLeftSingularVector,alg:MakeMaxRightSingularVector}
  indeed construct relatively inexact maximal left/right singular vector
  oracles, respectively.
\end{lemma}

\begin{proof}
  Let $A \in \RealMatrices{m}{n}$ and $\delta \in \OpenOpenInterval{0}{1}$
  be arbitrary, and let
  $\hat{u} \DefinedEqual \MaxLeftSingularVector(A, \delta)$
  be the output of the $\MaxLeftSingularVector$ procedure constructed by
  \cref{alg:MakeMaxLeftSingularVector}.
  By our definitions,
  $\hat{u} = \MaxEigenVector_{1 / 2}(A A\Transpose, \delta)$.
  Since $\MaxEigenVector_{1 / 2}$ is a relatively inexact stochastic maximal
  eigenvector oracle of degree~$1 / 2$, this means that $\Norm{\hat{u}} = 1$
  (a.s.) and
  \[
    \Expectation \Norm{A\Transpose \hat{u}}
    =
    \Expectation[\InnerProduct{A A\Transpose \hat{u}}{\hat{u}}^{1 / 2}]
    \geq
    (1 - \delta) [\MaxEigenValue(A A\Transpose)]^{1 / 2}
    =
    (1 - \delta) \MaxSingularValue(A).
  \]
  This proves that $\hat{u}$ is a $\delta$-relatively inexact stochastic
  maximal left singular vector of~$A$.

  Thus, \cref{alg:MakeMaxLeftSingularVector} is indeed correct.
  The justification for \cref{alg:MakeMaxRightSingularVector} is completely
  similar.
\end{proof}

\begin{definition}[Relatively inexact maximal pair of singular vectors]
  Given a matrix $A \in \RealMatrices{m}{n}$,
  a scalar $\delta \in \OpenOpenInterval{0}{1}$,
  and random vectors~$\hat{u} \in \RealField^m$ and $\hat{v} \in \RealField^n$,
  we say that $(\hat{u}, \hat{v})$ is a
  \emph{$\delta$-relatively inexact stochastic maximal pair of singular vectors}
  of~$A$ iff $\Norm{\hat{u}} = \Norm{\hat{v}} = 1$ (a.s.) and
  \[
    \Expectation \InnerProduct{A \hat{v}}{\hat{u}}
    \geq
    (1 - \delta) \MaxSingularValue(A).
  \]
  A \emph{relatively inexact stochastic maximal pair-of-singular-vectors oracle}
  is a procedure $\MaxSingularVectorPair$ that accepts
  a matrix~$A \in \RealMatrices{m}{n}$
  and a scalar~$\delta \in \OpenOpenInterval{0}{1}$,
  and returns a pair of random
  vectors~$(\hat{u}, \hat{v}) = \MaxSingularVectorPair(A, \delta)$
  such that $(\hat{u}, \hat{v})$ is a $\delta$-relatively inexact stochastic
  maximal singular vector pair.
\end{definition}

To construct a pair of maximal singular vectors, we may first compute a left
singular vector and then construct the right one from it.
Alternatively, we can compute the right singular vector and then construct
the left one from it.

\begin{center}
  \begin{minipage}[t]{0.45\linewidth}
    \begin{SimpleAlgorithm}[
      label = alg:MakeMaxSingularVectorPairFromLeft,
      width = \linewidth
    ]
      \begin{AlgorithmGroup}[Signature]
        $
          \MaxSingularVectorPair
          =
          \MakeMaxSingularVectorPairFromLeft(\MaxLeftSingularVector)
        $
      \end{AlgorithmGroup}
      \AlgorithmGroupSeparator

      \begin{AlgorithmGroup}[Input]
        $\MaxLeftSingularVector$---relatively inexact maximal left singular
        vector oracle.
      \end{AlgorithmGroup}
      \begin{AlgorithmGroup}[Output]
        $\MaxSingularVectorPair$---relatively inexact maximal
        pair-of-singular-vectors oracle.
      \end{AlgorithmGroup}
      \AlgorithmGroupSeparator

      \begin{AlgorithmGroup}
        \begin{AlgorithmSteps}
          \AlgorithmStep
            Return the function
            $\MaxSingularVectorPair = \MaxSingularVectorPair(A, \delta)$
            defined by the following computation:
            \begin{AlgorithmSteps}
              \AlgorithmStep
                $\hat{u} \DefinedEqual \MaxLeftSingularVector(A, \delta)$.
              \AlgorithmStep
                $\hat{v}' \DefinedEqual A\Transpose \hat{u}$.
              \AlgorithmStep
                $\hat{v} \DefinedEqual \frac{\hat{v}'}{\Norm{\hat{v}'}}$.
              \AlgorithmStep
                Return $(\hat{u}, \hat{v})$.
            \end{AlgorithmSteps}
        \end{AlgorithmSteps}
      \end{AlgorithmGroup}
    \end{SimpleAlgorithm}
  \end{minipage}
  \hspace{1em}
  \begin{minipage}[t]{0.47\linewidth}
    \begin{SimpleAlgorithm}[
      label = alg:MakeMaxSingularVectorPairFromRight,
      width = \linewidth
    ]
      \begin{AlgorithmGroup}[Signature]
        $
          \MaxSingularVectorPair
          =
          \MakeMaxSingularVectorPairFromRight(\MaxRightSingularVector)
        $
      \end{AlgorithmGroup}
      \AlgorithmGroupSeparator

      \begin{AlgorithmGroup}[Input]
        $\MaxRightSingularVector$---relatively inexact maximal right singular
        vector oracle.
      \end{AlgorithmGroup}
      \begin{AlgorithmGroup}[Output]
        $\MaxSingularVectorPair$---relatively inexact maximal
        pair-of-singular-vectors oracle.
      \end{AlgorithmGroup}
      \AlgorithmGroupSeparator

      \begin{AlgorithmGroup}
        \begin{AlgorithmSteps}
          \AlgorithmStep
            Return the function
            $\MaxSingularVectorPair = \MaxSingularVectorPair(A, \delta)$
            defined by the following computation:
            \begin{AlgorithmSteps}
              \AlgorithmStep
                $\hat{v} \DefinedEqual \MaxRightSingularVector(A, \delta)$.
              \AlgorithmStep
                $\hat{u}' \DefinedEqual A \hat{v}$.
              \AlgorithmStep
                $\hat{u} \DefinedEqual \frac{\hat{u}'}{\Norm{\hat{u}'}}$.
              \AlgorithmStep
                Return $(\hat{u}, \hat{v})$.
            \end{AlgorithmSteps}
        \end{AlgorithmSteps}
      \end{AlgorithmGroup}
    \end{SimpleAlgorithm}
  \end{minipage}
\end{center}

\begin{lemma}
  Both
  \cref{%
    alg:MakeMaxSingularVectorPairFromLeft,%
    alg:MakeMaxSingularVectorPairFromRight%
  }
  are correct in the sense that their output is indeed a relatively inexact
  stochastic maximal pair-of-singular-vectors oracle.
\end{lemma}

\begin{proof}
  We only prove the claim for \cref{alg:MakeMaxSingularVectorPairFromLeft},
  as the proof for \cref{alg:MakeMaxSingularVectorPairFromRight} is completely
  similar.

  Let $A \in \RealMatrices{m}{n}$ and $\delta \in \OpenOpenInterval{0}{1}$
  be arbitrary, and let $(\hat{u}, \hat{v}) = \MaxSingularVectorPair(A, \delta)$
  be the output of the $\MaxSingularVectorPair$ procedure constructed by
  \cref{alg:MakeMaxSingularVectorPairFromLeft}.
  From our definition of~$\hat{u}$ and the fact that $\MaxLeftSingularVector$
  is a relatively inexact maximal left singular vector oracle, it follows that
  $\Norm{\hat{u}} = 1$ (a.s.) and
  \[
    \Expectation \Norm{A\Transpose \hat{u}}
    \geq
    (1 - \delta) \MaxSingularValue(A).
  \]
  On the other hand, from the definitions of~$\hat{v}'$ and~$\hat{v}$,
  it is clear that $\Norm{\hat{v}} = 1$ and
  \[
    \InnerProduct{A \hat{v}}{\hat{u}}
    =
    \InnerProduct{\hat{v}}{\hat{v}'}
    =
    \Norm{\hat{v}'}
    =
    \Norm{A\Transpose \hat{u}}.
  \]
  Combining the above two displays, we conclude that $(\hat{u}, \hat{v})$
  is a $\delta$-relatively inexact stochastic maximal pair of singular vectors
  of~$A$.
\end{proof}
\subsubsection{Power Method}

The simplest oracle for $\MaxEigenVector(A, \delta)$ is given by the
\emph{Power method}.
The idea is to choose a random starting vector $\xi \in \RealField^n$
and then compute
\[
  \hat{u}_p \DefinedEqual \frac{A^p \xi}{\Norm{A^p \xi}}
\]
for a sufficiently large integer power~$p \geq 1$.

For numerical stability, this algorithm is typically implemented in a slightly
different form where the normalization is performed after each matrix-vector
multiplication.

\begin{SimpleAlgorithm}[
  title = {Power Method},
  label = {alg:PowerMethod},
  width = 0.45\linewidth
]
  \begin{AlgorithmGroup}[Signature]
    $\hat{u} = \PowerMethod(A, p)$.
  \end{AlgorithmGroup}
  \AlgorithmGroupSeparator

  \begin{AlgorithmGroup}[Input]
    Matrix~$A \in \SymmetricPsdMatrices{n}$,
    number of matrix-vector products~$p$ [integer $\geq 1$].
  \end{AlgorithmGroup}
  \begin{AlgorithmGroup}[Output]
    Random vector~$\hat{u} \in \UnitSphere{n - 1}$.
  \end{AlgorithmGroup}
  \AlgorithmGroupSeparator

  \begin{AlgorithmGroup}
    \begin{AlgorithmSteps}
      \AlgorithmStep
        Sample $u_0 \SampledFrom \UniformDistribution{\UnitSphere{n - 1}}$.
      \AlgorithmStep
        Iterate for $k = 1, \ldots, p$:
        \begin{AlgorithmSteps}
          \AlgorithmStep
            $\hat{u}_k' \DefinedEqual A \hat{u}_{k - 1}$.
          \AlgorithmStep
            $\hat{u}_k \DefinedEqual \frac{\hat{u}_k'}{\Norm{\hat{u}_k'}}$.
        \end{AlgorithmSteps}
      \AlgorithmStep
        Return $\hat{u}_p$.
    \end{AlgorithmSteps}
  \end{AlgorithmGroup}
\end{SimpleAlgorithm}

The standard convergence result about the Power Method is as follows.

\begin{theorem}[Theorem~3.1 in~\cite{Kuczyski-Wozniakowski-92-EstimatingLargest}]
  \label{th:ConvergenceRateOfPowerMethod}
  Suppose that the input in \cref{alg:PowerMethod} is such that
  $n \geq 8$ and $p \geq 2$.
  Then, the output vector~$\hat{u}$ is a $\delta_p$-relatively inexact
  stochastic unit eigenvector of~$A$, where
  \[
    \delta_p \DefinedEqual 0.871 \frac{\ln n}{p}.
  \]
\end{theorem}

From the above theorem, we know how to,
given an accuracy~$\delta \in \OpenOpenInterval{0}{1}$, choose the number of
iterations~$p$ sufficient to guarantee that the output of the Power method
is a $\delta$-relatively inexact stochastic unit eigenvector of~$A$.
This allows us to build an oracle for~$\MaxEigenVector(A, \delta)$.

\begin{SimpleAlgorithm}[
  title = {Power Oracle for Maximal Eigenvector},
  label = alg:PowerOracle,
  width = 0.5\linewidth
]
  \begin{AlgorithmGroup}[Signature]
    $\hat{u} = \MaxEigenVectorPower(A, \delta)$.
  \end{AlgorithmGroup}
  \AlgorithmGroupSeparator

  \begin{AlgorithmGroup}[Input]
    Matrix~$A \in \SymmetricPsdMatrices{n}$,
    accuracy~$\delta \in \OpenOpenInterval{0}{1}$.
  \end{AlgorithmGroup}
  \begin{AlgorithmGroup}[Output]
    Random vector~$\hat{u} \in \RealField^n$ such that
    $\hat{u}$ is a $\delta$-relatively inexact stochastic maximal eigenvector
    of~$A$.
  \end{AlgorithmGroup}
  \AlgorithmGroupSeparator

  \begin{AlgorithmGroup}
    \begin{AlgorithmSteps}
      \AlgorithmStep
        Set $p \DefinedEqual \Ceil{0.871 \frac{\ln n}{\delta}}$.
      \AlgorithmStep
        Return $\PowerMethod(A, p)$.
    \end{AlgorithmSteps}
  \end{AlgorithmGroup}
\end{SimpleAlgorithm}

\begin{theorem}
  For any $n \geq 8$, \cref{alg:PowerOracle} indeed returns
  a $\delta$-relatively inexact stochastic maximal eigenvector of~$A$.
  The running time of the method is
  \[
    \Ceil[\Big]{\frac{\ln n}{\delta}}
    \times
    \RunTime\bigl( \MatrixVectorProduct(A) \bigr)
    +
    \BigO\Bigl( \frac{n \ln n}{\delta} \Bigr),
  \]
  where $\RunTime\bigl( \MatrixVectorProduct(A) \bigr)$ is the running time of
  a matrix-vector multiplication for~$A$.
\end{theorem}

\begin{proof}
  The fact that $\hat{u}$ is indeed a $\delta$-relatively inexact stochastic
  maximal eigenvector of~$A$ follows from \cref{th:ConvergenceRateOfPowerMethod}
  (note that, since $n \geq 8$ and $\delta \in \OpenOpenInterval{0}{1}$,
  we have $p \geq \Ceil{0.871 \cdot \ln 8} = \Ceil{1.81\ldots} = 2$).

  To establish the complexity bound, note that, at each iteration
  of the loop, only one matrix-vector multiplication is performed.
  Therefore, the total number of matrix vector multiplications is
  \[
    p
    =
    \Ceil[\Big]{0.871 \frac{\ln n}{\delta}}
    \leq
    \Ceil[\Big]{\frac{\ln n}{\delta}}.
  \]
  The number of auxiliary operations at each iteration is~$\BigO(n)$.
  Therefore, the extra running time complexity is
  \[
    \BigO(n p)
    \leq
    \BigO\Bigl( n \Bigl( \frac{\ln n}{\delta} + 1 \Bigr) \Bigr)
    =
    \BigO\Bigl( \frac{n \ln n}{\delta} \Bigr)
  \]
  since $\frac{\ln n}{\delta} \geq 1$ for $n \geq 8$.
\end{proof}
\subsubsection{Lanczos Algorithm}

The Lanczos method chooses a random vector~$\xi \in \UnitSphere{n - 1}$
and then searches for a maximizer of the Rayleigh quotient in the $p$-th Krylov
subspace generated by~$\xi$:
\[
  \hat{v}_p
  \in
  \ArgmaxSet \SetBuilder[\big]{
    \InnerProduct{A v}{v}
  }{
    v \in \KrylovSubspace_p(A, \xi) \cap \UnitSphere{n - 1}
  },
\]
where
\[
  \KrylovSubspace_p(A, \xi)
  \DefinedEqual
  \LinearHull\Set{\xi, A \xi, \ldots, A^p \xi}.
\]
We refer to any vector~$\hat{v}_p$ satisfying the above inclusion as
a \emph{Lanczos maximal eigenvector} of order~$p$ for $(A, \xi)$.
Such a vector can be found very efficiently by using only $\BigO(p)$
matrix-vector products with matrix~$A$.

The key step is to first find an orthonormal basis for the Krylov
subspace~$\KrylovSubspace_p(A, \xi)$---a
matrix~$Q \in \RealMatrices{n}{(p + 1)}$ with $Q\Transpose Q = I$
(identity matrix) and $\Image Q = \KrylovSubspace_p(A, \xi)$---in which
$A$ is tridiagonal:
\[
  Q\Transpose A Q = \TridiagonalMatrix(\alpha, \beta),
\]
where $\TridiagonalMatrix(\alpha, \beta)$ the $(p + 1) \times (p + 1)$
symmetric tridiagonal matrix with some (known)
vector~$\alpha \in \RealField^{p + 1}$ on the main diagonal and
some (known) vector~$\beta \in \RealField^p$ on the sub- and super diagonals.
In what follows, we refer to such an $(\alpha, \beta, Q)$ as a \emph{Lanczos
tridiagonalization triple} of order~$p$ for $(A, \xi)$.
This triple can be constructed using the following algorithm.

\begin{SimpleAlgorithm}[
  title = {Lanczos Tridiagonalization},
  label = alg:LanczosTridiagonalization,
  width = 0.8\linewidth
]
  \begin{AlgorithmGroup}[Signature]
    $(\alpha, \beta, Q) = \LanczosTridiagonalization(A, \xi, p)$.
  \end{AlgorithmGroup}
  \AlgorithmGroupSeparator

  \begin{AlgorithmGroup}[Input]
    Matrix~$A \in \SymmetricMatrices{n}$,
    vector~$\xi \in \UnitSphere{n - 1}$,
    number of iterations~$p \geq 1$.
  \end{AlgorithmGroup}
  \begin{AlgorithmGroup}[Output]
    Vectors~$\alpha \in \RealField^{p + 1}$, $\beta \in \RealField^p$,
    and matrix~$Q \in \RealMatrices{n}{(p + 1)}$ such that $(\alpha, \beta, Q)$
    is a Lanczos tridiagonalization triple of order~$p$ for~$(A, \xi)$.

  \end{AlgorithmGroup}
  \AlgorithmGroupSeparator

  \begin{AlgorithmGroup}
    \begin{AlgorithmSteps}
      \AlgorithmStep
        Set $q_0 \DefinedEqual \xi$.
        Compute
        $q_0' \DefinedEqual A q_0$,
        $\alpha_0 \DefinedEqual \InnerProduct{q_0'}{q_0}$,
        and $r_0 \DefinedEqual q_0' - \alpha_0 q_0$.
      \AlgorithmStep
        Iterate for $k = 0, \ldots, p - 1$:
        \begin{AlgorithmSteps}
          \AlgorithmStep
            Compute
            $\beta_k \DefinedEqual \Norm{r_k}$,
            $q_{k + 1} \DefinedEqual r_k / \beta_k$,
            and $q_{k + 1}' \DefinedEqual A q_k$.
          \AlgorithmStep
            Compute
            $\alpha_{k + 1} \DefinedEqual \InnerProduct{q_{k + 1}'}{q_k}$
            and
            $r_{k + 1} \DefinedEqual q_{k + 1}' - \alpha_{k + 1} q_{k + 1} - \beta_k q_k$.
        \end{AlgorithmSteps}
      \AlgorithmStep
        Return
        $\alpha \DefinedEqual (\alpha_0, \ldots, \alpha_p)$,
        $\beta \DefinedEqual (\beta_0, \ldots, \beta_{p - 1})$,
        $Q \DefinedEqual [q_0, \ldots, q_p]$.
    \end{AlgorithmSteps}
  \end{AlgorithmGroup}
\end{SimpleAlgorithm}

\begin{theorem}[Theorem~10.1.1 in~\cite{Golub-vanLoan-13-MatrixComputations}]
  \Cref{alg:LanczosTridiagonalization} is correct in the sense that it indeed
  produces a Lanczos tridiagonalization triple.
\end{theorem}

Combining \cref{alg:LanczosTridiagonalization} with any \emph{exact}
algorithm $\MaxEigenVectorOfTridiagonalMatrix$ for computing a maximal
unit eigenvector of a tridiagonal matrix, we get the Lanczos method for
computing a leading eigenvector of~$A$.

\begin{SimpleAlgorithm}[
  title = {Lanczos Algorithm},
  label = alg:LanczosAlgorithm,
  width = 0.7\linewidth,
]
  \begin{AlgorithmGroup}[Signature]
    $\hat{v} = \LanczosAlgorithm(A, p)$.
  \end{AlgorithmGroup}
  \AlgorithmGroupSeparator

  \begin{AlgorithmGroup}[Input]
    Matrix~$A \in \SymmetricMatrices{n}$,
    number of iterations~$p \geq 1$.
  \end{AlgorithmGroup}
  \begin{AlgorithmGroup}[Output]
    $\hat{v} \in \UnitSphere{n - 1}$---Lanczos maximal eigenvector of order~$p$
    for $(A, \xi)$ for a randomly sampled
    $\xi \DistributedAs \UniformDistribution{\UnitSphere{n - 1}}$.
  \end{AlgorithmGroup}
  \AlgorithmGroupSeparator

  \begin{AlgorithmGroup}
    \begin{AlgorithmSteps}
      \AlgorithmStep
        Sample $\xi \SampledFrom \UniformDistribution{\UnitSphere{n - 1}}$.
      \AlgorithmStep
        \label{step:LanczosTridiagonalization}
        Compute
        $(\alpha, \beta, Q) \DefinedEqual \LanczosTridiagonalization(A, \xi, p)$.
      \AlgorithmStep
        \label{step:ComputeMaxEigenVectorOfTridiagonalMatrix}
        Compute
        $\hat{x} \DefinedEqual \MaxEigenVectorOfTridiagonalMatrix(\alpha, \beta)$.
      \AlgorithmStep
        Return $\hat{v} \DefinedEqual Q \hat{x}$.
    \end{AlgorithmSteps}
  \end{AlgorithmGroup}
\end{SimpleAlgorithm}

In principle, we can use any algorithm in place of
$\MaxEigenVectorOfTridiagonalMatrix$.
However, for concreteness, we will assume that this is the standard
QR decomposition-based algorithm for computing eigenvectors and eigenvalues
of a symmetric tridiagonal matrix (see Section~8.3
in~\cite{Golub-vanLoan-13-MatrixComputations}).
The complexity of such an algorithm is~$\BigO(p^2)$, where
$p$ is the dimension of the tridiagonal matrix.
Even if the original dimension~$n$ of the matrix
was huge, the value of~$p$ is typically rather small, so this complexity is
affordable.

\begin{theorem}
  The vector~$\hat{v}$ returned by \cref{alg:LanczosAlgorithm} is indeed
  a Lanczos maximal eigenvector of order~$p$ for $(A, \xi)$.
\end{theorem}

\begin{proof}
  According to the guarantees of \cref{alg:LanczosTridiagonalization},
  at the end of \cref{step:LanczosTridiagonalization} of
  \cref{alg:LanczosAlgorithm}, the matrix~$Q$ is such that its columns form
  an orthonormal basis for $\KrylovSubspace_p(A, \xi)$
  and $Q\Transpose A Q = T$, where
  $T \DefinedEqual \TridiagonalMatrix(\alpha, \beta)$.
  Thus, any vector~$v$ from $\KrylovSubspace_p(A, \xi) \cap \UnitSphere{n - 1}$
  can be (uniquely) parameterized as $v = Q x$, where
  $x \in \UnitSphere{p}$.
  Consequently, $v$ is a Lanczos maximal eigenvector of order~$p$ for $(A, \xi)$
  iff $v = Q x$, where $x$ maximizes
  \[
    \InnerProduct{A (Q x)}{Q x}
    =
    \InnerProduct{Q\Transpose A Q x}{x}
    =
    \InnerProduct{T x}{x}
  \]
  over all $x \in \UnitSphere{p}$, or, equivalently, iff $x$ is a maximal unit
  eigenvector of~$T$.
  The claim now follows from the fact that
  $\hat{x}$ obtained at
  \cref{step:ComputeMaxEigenVectorOfTridiagonalMatrix} is indeed a maximal unit
  eigenvector of~$T$ and $\hat{v} = Q \hat{x}$.
\end{proof}

The classical convergence bound for the Lanczos algorithm is as follows.

\begin{theorem}[Theorem~3.2 in~\cite{Kuczyski-Wozniakowski-92-EstimatingLargest}]
  \label{th:ConvergenceRateForLanczosAlgorithm}
  Suppose that the input in \cref{alg:LanczosAlgorithm} is such that
  $A$ is positive semidefinite, $n \geq 8$ and $p \geq 3$.
  Then, the output $\hat{v}$ generated by the method is a
  $\delta_p$-relatively inexact stochastic unit eigenvector of~$A$, where
  \[
    \delta_p \DefinedEqual 2.575 \Bigl( \frac{\ln n}{p} \Bigr)^2.
  \]
\end{theorem}

Using the previous bound, we can now select~$p$ for any required~$\delta$
and thus construct the Lanczos oracle for $\MaxEigenVector(A, \delta)$.

\begin{SimpleAlgorithm}[
  title = {Lanczos Oracle for Computing Maximal Eigenvector},
  label = alg:LanczosOracle,
  width = 0.6\linewidth,
]
  \begin{AlgorithmGroup}[Signature]
    $\hat{v} = \MaxEigenVectorLanczos(A, \delta)$.
  \end{AlgorithmGroup}
  \AlgorithmGroupSeparator

  \begin{AlgorithmGroup}[Input]
    Matrix~$A \in \SymmetricMatrices{n}$,
    accuracy~$\delta \in \OpenOpenInterval{0}{1}$.
  \end{AlgorithmGroup}
  \begin{AlgorithmGroup}[Output]
    $\hat{v} \in \UnitSphere{n - 1}$---$\delta$-relatively inexact
    stochastic unit eigenvector of~$A$.
  \end{AlgorithmGroup}
  \AlgorithmGroupSeparator

  \begin{AlgorithmGroup}
    \begin{AlgorithmSteps}
      \AlgorithmStep
        Compute $p \DefinedEqual \Ceil{1.605 \frac{\ln n}{\sqrt{\delta}}}$.
      \AlgorithmStep
        Return $\LanczosAlgorithm(A, p)$.
    \end{AlgorithmSteps}
  \end{AlgorithmGroup}
\end{SimpleAlgorithm}

\begin{theorem}
  The output of \cref{alg:LanczosOracle} is indeed a $\delta$-relatively
  inexact stochastic unit eigenvector of~$A$.
  The total running time of the algorithm is
  \[
    \Bigl( 2 \frac{\ln n}{\sqrt{\delta}} + 1 \Bigr)
    \times
    \RunTime\bigl( \MatrixVectorProduct(A) \bigr)
    +
    \BigO\Bigl( \frac{n \ln n}{\sqrt{\delta}} \Bigr).
  \]
\end{theorem}

\begin{proof}
  The fact that $\hat{v}$ is indeed a $\delta$-relatively inexact stochastic
  unit eigenvector of~$A$ follows immediately from
  \cref{th:ConvergenceRateForLanczosAlgorithm}
  and the fact that $\sqrt{2.575} = 1.6046\ldots \leq 1.605$.

  To justify the time complexity, note that all matrix-vector products
  with~$A$ are performed only inside the call to
  $\LanczosTridiagonalization(A, p)$
  (one multiplication at each iteration plus one extra during initialization).
  The extra complexity at each iteration inside $\LanczosTridiagonalization$
  is $\BigO(n)$, and $\BigO(n)$ during initialization.
  The total extra complexity of the call to $\LanczosTridiagonalization(A, p)$
  is thus $\BigO(n p)$.

  Further, the complexity of $\MaxEigenVectorOfTridiagonalMatrix(\alpha, \beta)$
  is $\BigO(p^2) = \BigO(n p)$ since $p \leq n$
  (otherwise, the Lanczos algorithm terminates).
\end{proof}
\subsection{Maximal Eigenvalue of Symmetric Matrix}

\begin{lemma}
  Consider the function $\Map{f}{\SymmetricMatrices{n}}{\RealField}$
  defined by
  \[
    f(X) \DefinedEqual \MaxEigenValue(X).
  \]
  Let $\delta \in \OpenOpenInterval{0}{1}$.
  Consider the oracle $\hat{g}$ defined at each point
  $X \in \SymmetricMatrices{n}$ by
  \[
    \hat{g}(X) \DefinedEqual \hat{v} \hat{v}\Transpose,
    \qquad
    \hat{v} \DefinedEqual \MaxEigenVector(X, \delta).
  \]
  Then, $\hat{g}$ is a $\delta$-relatively inexact stochastic oracle for~$f$.
\end{lemma}

\begin{proof}
  Let $X, Y \in \SymmetricMatrices{n}$.
  Since $\hat{v}$ is a unit vector (a.s.), we have (a.s.)
  \[
    f(Y)
    =
    \MaxEigenValue(Y)
    \geq
    \InnerProduct{Y \hat{v}}{\hat{v}}
    =
    \InnerProduct{X \hat{v}}{\hat{v}}
    +
    \InnerProduct{(Y - X) \hat{v}}{\hat{v}}
    =
    \InnerProduct{X \hat{v}}{\hat{v}} + \InnerProduct{\hat{g}(X)}{Y - X}.
  \]
  Taking now expectations and using the fact that $\hat{v}$ is a
  $\delta$-relatively inexact eigenvector of~$X$, we get
  \[
    f(Y)
    =
    \Expectation f(Y)
    \geq
    \Expectation \InnerProduct{X \hat{v}}{\hat{v}}
    +
    \InnerProduct{\Expectation \hat{g}(X)}{Y - X}
    \geq
    (1 - \delta) f(X) + \InnerProduct{\Expectation \hat{g}(X)}{Y - X}.
    \qedhere
  \]
\end{proof}

Sometimes, we need to compute the oracle for the composition
of $\MaxEigenValue$ with an affine mapping:
\[
  f(x) = \MaxEigenValue(A x + C), \qquad x \in \RealField^d
\]
where $\Map{A}{\RealField^d}{\SymmetricMatrices{n}}$ is the linear operator
\[
  A x \DefinedEqual \sum_{i = 1}^d x_i A_i,
\]
and $A_1, \ldots, A_d, C \in \SymmetricMatrices{n}$.
In this case, our oracle is
\[
  \hat{g}(X) = A\Adjoint \hat{G}(A x + C),
\]
where $\hat{G}$ is the standard oracle for $\MaxEigenValue$:
\[
  \hat{G}(Y) = \hat{u} \hat{u}\Transpose,
  \qquad
  \hat{u} \DefinedEqual \MaxEigenVector(Y, \delta),
\]
and $\Map{A\Adjoint}{\SymmetricMatrices{n}}{\RealField^d}$ is the adjoint
operator
\[
  A\Adjoint G = (\InnerProduct{A_i}{G})_{i = 1}^d.
\]
In this case, we can evaluate $\hat{g}(x)$ without forming any intermediate
matrices:
\[
  \hat{g}(x) = (\InnerProduct{A_i \hat{u}}{\hat{u}})_{i = 1}^d,
  \qquad
  \hat{u} \DefinedEqual \MaxEigenVector(A x + C, \delta).
\]
The complexity of this operation is just extra $d$ matrix-vector multiplications
of $A_i$ by $\hat{u}$.

\subsection{Maximal Singular Value}

\begin{lemma}
  Consider the function $\Map{f}{\RealMatrices{m}{n}}{\RealField}$
  defined by
  \[
    f(X) \DefinedEqual \MaxSingularValue(X).
  \]
  Let $\delta \in \OpenOpenInterval{0}{1}$.
  Consider the oracle $\hat{g}$ defined at each point
  $X \in \RealMatrices{m}{n}$ by
  \[
    \hat{g}(X) \DefinedEqual \hat{u} \hat{v}\Transpose,
    \qquad
    (\hat{u}, \hat{v}) \DefinedEqual \MaxSingularVectorPair(X, \delta).
  \]
  Then, $\hat{g}$ is a $\delta$-relatively inexact stochastic oracle for~$f$.
\end{lemma}

\begin{proof}
  Let $X, Y \in \RealMatrices{m}{n}$.
  Since $\hat{u}$ and $\hat{v}$ are unit vectors (a.s.), we have (a.s.)
  \[
    f(Y)
    =
    \MaxSingularValue(Y)
    \geq
    \InnerProduct{Y \hat{v}}{\hat{u}}
    =
    \InnerProduct{X \hat{v}}{\hat{u}}
    +
    \InnerProduct{(Y - X) \hat{v}}{\hat{u}}
    =
    \InnerProduct{X \hat{v}}{\hat{u}} + \InnerProduct{\hat{g}(X)}{Y - X}.
  \]
  Taking now expectations and using the fact that $(\hat{u}, \hat{v})$ is a
  $\delta$-relatively inexact pair of singular vectors of~$X$, we get
  \[
    f(Y)
    =
    \Expectation f(Y)
    \geq
    \Expectation \InnerProduct{X \hat{v}}{\hat{u}}
    +
    \InnerProduct{\Expectation \hat{g}(X)}{Y - X}
    \geq
    (1 - \delta) f(X) + \InnerProduct{\Expectation \hat{g}(X)}{Y - X}.
    \qedhere
  \]
\end{proof}
\subsection{Squared Spectral Norm}
\label{sec:SquaredSpectralNorm}

\begin{lemma}
  \label{th:OracleForSquaredSpectralNorm}
  Consider the function $\Map{f}{\RealMatrices{m}{n}}{\RealField}$
  defined by
  \[
    f(X)
    \DefinedEqual
    \SchattenNorm{X}{\infty}^2
    =
    [\MaxSingularValue(X)]^2.
  \]
  Let $\delta \in \OpenOpenInterval{0}{1}$.
  Consider the oracle $\hat{g}$ defined at each
  point $X \in \RealMatrices{m}{n}$ by
  \[
    \hat{g}(X) \DefinedEqual 2 \hat{u} \hat{u}\Transpose X,
    \qquad
    \hat{u} \DefinedEqual \MaxEigenVector(X X\Transpose, \delta).
  \]
  Then, $\hat{g}$ is a $\delta$-relatively inexact stochastic oracle for~$f$.
\end{lemma}

\begin{proof}
  Let $X, Y \in \RealMatrices{n}{m}$, and let $\hat{u}$ be the random
  vector from the definition of~$\hat{g}(X)$.
  Since $\hat{u}$ is the output of a relatively inexact maximal eigenvector
  oracle~$\MaxEigenVector$, we have $\Norm{\hat{u}} = 1$ (a.s.).
  Consequently (a.s.),
  \[
    f(Y)
    =
    [\MaxSingularValue(Y)]^2
    =
    \MaxEigenValue(Y Y\Transpose)
    \geq
    \InnerProduct{Y Y\Transpose \hat{u}}{\hat{u}}
    =
    \InnerProduct{X X\Transpose \hat{u}}{\hat{u}}
    +
    \InnerProduct{(Y Y\Transpose - X X\Transpose) \hat{u}}{\hat{u}}.
  \]
  Note that
  \begin{align*}
    Y Y\Transpose - X X\Transpose
    &=
    (Y - X) Y\Transpose + X (Y - X)\Transpose
    \\
    &=
    (Y - X) X\Transpose + X (Y - X)\Transpose + (Y - X) (Y - X)\Transpose.
  \end{align*}
  Hence (a.s.),
  \begin{align*}
    f(Y)
    &\geq
    \InnerProduct{X X\Transpose \hat{u}}{\hat{u}}
    +
    2 \InnerProduct{(Y - X) X\Transpose \hat{u}}{\hat{u}}
    +
    \Norm{(Y - X)\Transpose \hat{u}}^2
    \\
    &\geq
    \InnerProduct{X X\Transpose \hat{u}}{\hat{u}}
    +
    \InnerProduct{\hat{g}(X)}{Y - X}.
  \end{align*}
  Taking now the expectation on both sides and using the fact that
  $\hat{u}$ is a $\delta$-relatively inexact stochastic eigenvector
  of~$X X\Transpose$, we obtain
  \begin{align*}
    f(Y)
    &=
    \Expectation f(Y)
    \geq
    (1 - \delta) \MaxEigenValue(X X\Transpose)
    +
    \InnerProduct{\Expectation \hat{g}(X)}{Y - X}
    \\
    &=
    (1 - \delta) f(X) + \InnerProduct{\Expectation \hat{g}(X)}{Y - X}.
    \qedhere
  \end{align*}
\end{proof}

Note that the above oracle is relatively bounded w.r.t.\ the function.

\begin{lemma}
  \label{th:OracleForSquaredSpectralNormIsRelativeBounded}
  The oracle~$\hat{g}$ from \cref{th:OracleForSquaredSpectralNorm}
  is $2$-relatively bounded (in the standard Frobenius norm) w.r.t.\ the
  function~$\SchattenNorm{\cdot}{\infty}^2$.
\end{lemma}

\begin{proof}
  Let $X \in \RealMatrices{m}{n}$, and let $\hat{u}$ be the random vector
  from the definition of~$\hat{g}(X)$.
  Since $\Norm{\hat{u}} = 1$ (a.s.), we have (a.s.)
  \begin{align*}
    \FrobeniusNorm{\hat{g}(X)}^2
    &=
    4 \InnerProduct{\hat{u} \hat{u}\Transpose X}{\hat{u} \hat{u}\Transpose X}
    =
    4 \InnerProduct{X X\Transpose \hat{u}}{\hat{u}}
    \\
    &\leq
    4 \MaxEigenValue(X X\Transpose)
    =
    4 [\MaxSingularValue(X)]^2
    =
    4 f(X).
  \end{align*}
  Consequently, $\Expectation[\FrobeniusNorm{\hat{g}(X)}^2] \leq 4 f(X)$.
\end{proof}
  \section{Spectral Linear Regression}
\label{sec:SpectralLinearRegression}
\UsingNamespace{SpectralLinearRegression}

Consider the problem of linear approximation of a given
matrix~$C \in \RealMatrices{n}{m}$ by a given collection of
matrices $A_1, \ldots, A_d \in \RealMatrices{n}{m}$
w.r.t.\ the matrix infinity norm:
\begin{equation}
  \LocalLabel{eq:Problem}
  f^*
  \DefinedEqual
  \min_{x \in \RealField^d} f(x),
  \qquad
  f(x) \DefinedEqual \SchattenNorm[\Big]{\sum_{i = 1}^d x_i A_i - C}{\infty}.
\end{equation}
Note that problem~\eqref{\LocalName{eq:Problem}} is very similar to a classical
linear regression problem.
The only difference is that we measure the residual between matrices
in the \emph{spectral} norm instead of the Frobenius one.
In view of this analogy, we refer to problem~\eqref{\LocalName{eq:Problem}} as
a \emph{spectral linear regression} problem.

In what follows, without loss of generality, we assume that $n \leq m$
(otherwise, we can simply transpose all matrices).

We are going to find an approximate solution to
problem~\eqref{\LocalName{eq:Problem}} in \emph{relative scale}.
For this, however, it will be convenient to first transform this problem
into an equivalent one by squaring the objective function:
\begin{equation}
  \LocalLabel{eq:SquaredProblem}
  (f^*)^2
  =
  \min_{x \in \RealField^d} f^2(x),
  \qquad
  f^2(x) = \SchattenNorm{A x - C}{\infty}^2,
\end{equation}
where $\Map{A}{\RealField^d}{\RealMatrices{n}{m}}$ is the linear operator
\begin{equation}
  \LocalLabel{eq:LinearOperator}
  A x \DefinedEqual \sum_{i = 1}^d x_i A_i,
  \qquad
  x \in \RealField^d.
\end{equation}

Let us show that problem~\eqref{\LocalName{eq:SquaredProblem}} fits
the setting from \cref{sec:CompositionWithAffineMapping}.

First, note that \cref{CompositionWithAffineMapping::as:ClosednessOfSetImage}
is satisfied as problem~\eqref{\LocalName{eq:SquaredProblem}} is unconstrained.

Further, let us equip the space~$\RealMatrices{n}{m}$ with the standard
Frobenius norm:
\[
  \Norm{X} \DefinedEqual \FrobeniusNorm{X},
  \qquad
  X \in \RealMatrices{n}{m}.
\]
In the notation of \cref{sec:CompositionWithAffineMapping},
this is the Euclidean seminorm~$\RelativeNorm{\cdot}{B_1}$ with
$B_1 = I$ (identity operator in $\RealMatrices{n}{m}$).

Clearly, we have
\[
  f^2(x) = F(A x - C),
  \qquad
  \forall x \in \RealField^d,
\]
where $\Map{F}{\RealMatrices{n}{m}}{\RealField}$ is the squared spectral norm:
\[
  F(Y) \DefinedEqual \SchattenNorm{Y}{\infty}^2.
\]
Note that, for any $X \in \RealMatrices{n}{m}$, we have
$F(Y) = [\MaxSingularValue(X)]^2 \geq \frac{1}{n} \FrobeniusNorm{X}^2$
(recall that $n \leq m$).
Thus, the function~$F$ is consistent with the norm~$\Norm{\cdot}$ with the
following parameters:
\begin{equation}
  \LocalLabel{eq:ConsistencyParameters}
  \gamma_0 \DefinedEqual \frac{1}{n},
  \qquad
  \hat{Y}_0 \DefinedEqual 0.
\end{equation}
From \cref{CompositionWithAffineMapping::th:ConsistencyConditionsArePreserved},
it follows that, w.r.t.\ the seminorm~$\RelativeNorm{\cdot}{B}$ induced by the
Gram matrix
\begin{equation}
  \label{eq:GramMatrix}
  B = A\Adjoint A = (\InnerProduct{A_i}{A_j})_{i, j = 1}^d,
\end{equation}
the function~$f$ is also consistent with parameters~$\gamma_0$ and
\[
  \hat{x}_0 = T(0, -A\Adjoint C),
\]
where $T(\cdot, \cdot)$ is the gradient step operation:
\[
  T(\bar{x}, g)
  \DefinedEqual
  \argmin_{x \in \RealField^d} \Bigl\{
    \InnerProduct{g}{x} + \frac{1}{2} \RelativeNorm{x - \bar{x}}{B}^2
  \Bigr\},
  \qquad
  \bar{x} \in \RealField^d, \
  g \in (\ker B)\OrthogonalComplement.
\]
Note that the point~$T \DefinedEqual T(\bar{x}, g)$ can be computed by solving
the following linear system (which is guaranteed to be solvable):
\[
  B (T - \bar{x}) = -g.
\]

It remains to specify an (efficiently computable) relatively inexact stochastic
oracle for the function~$F$.
According to our discussion in
\cref{sec:CompositionWithAffineMapping,sec:SquaredSpectralNorm},
a good choice is the oracle $\hat{g}(x) = A\Adjoint \hat{G}(A x - C)$,
where $\hat{G}$ is the standard $(\Delta / 2)$-relatively inexact oracle for the
squared spectral norm from \cref{th:OracleForSquaredSpectralNorm}
(induced by our choice of a relatively inexact stochastic maximal eigenvector
oracle~$\MaxEigenVector$), and $\Delta \in \OpenOpenInterval{0}{1}$
is a fixed constant (to be specified later).

Recall from \cref{th:OracleForSquaredSpectralNormIsRelativeBounded} that
the oracle~$\hat{G}$ is $2$-relatively bounded (in the standard Frobenius
norm) w.r.t.\ the function~$F$.
Therefore, according to
\cref{CompositionWithAffineMapping::th:ConsistencyConditionsArePreserved},
the oracle~$\hat{g}$ is also $2$-relatively bounded but w.r.t.\
the function~$f$ and in the seminorm~$\RelativeNorm{\cdot}{B}$ induced by the
Gram matrix~\eqref{eq:GramMatrix}.

Applying now \cref{GradientMethodWithRelativelyInexactStochasticOracle::alg:Algorithm}
with the oracle~$\hat{g}$, initial point~$x_0 = \hat{x}_0$
constant~$L = 2$, accuracy~$\delta' = \Delta / 2$
and step sizes
\begin{equation}
  \LocalLabel{eq:StepSizes}
  a_k = \frac{L}{2 \delta'}, \qquad k \geq 0,
\end{equation}
we conclude, from
\cref{GradientMethodWithRelativelyInexactStochasticOracle::th:ConvergenceRateForSimplerChoiceOfStepSizes},
that, once the number of iterations~$N$ performed by the algorithm becomes
sufficiently large, namely,
\begin{equation}
  \LocalLabel{eq:PreliminaryNumberOfIterations}
  N
  \geq
  \frac{2 L}{\gamma_0 (\delta')^2}
  =
  \frac{16}{\gamma_0 \Delta^2}
  =
  \frac{16 n}{\Delta^2}
\end{equation}
(see \cref{\LocalName{eq:ConsistencyParameters}}),
the point~$x_N \in \RealField^d$ constructed by the algorithm is guaranteed
to be a $\Delta$-relatively inexact solution to
problem~\eqref{\LocalName{eq:SquaredProblem}}:
\begin{equation}
  \LocalLabel{eq:GuaranteeForSquaredProblem}
  (1 - \Delta) \Expectation [f^2(x_N)] \leq (f^*)^2.
\end{equation}

Recall, however, that our initial problem was~\eqref{\LocalName{eq:Problem}},
not~\eqref{\LocalName{eq:SquaredProblem}}.
Let us therefore see what guarantees we have for the point~$x_N$
in terms of our initial problem.
Using Jensen's inequality in \cref{\LocalName{eq:GuaranteeForSquaredProblem}},
we get
\[
  \sqrt{1 - \Delta} \,
  \Expectation f(x_N)
  \leq
  \sqrt{(1 - \Delta) \Expectation [f^2(x_N)]}
  \leq
  f^*.
\]
Hence, for any given $\delta \in \OpenOpenInterval{0}{1}$, choosing
\begin{equation}
  \LocalLabel{eq:RelativeAccuracyForSquaredProblem}
  \Delta
  \DefinedEqual
  1 - (1 - \delta)^2
  =
  (2 - \delta) \delta
  \quad
  (\in \OpenOpenInterval{0}{1}),
\end{equation}
we can guarantee that the point~$x_N$ is a $\delta$-relatively inexact solution
to our original problem~\eqref{\LocalName{eq:Problem}},
\[
  (1 - \delta) \Expectation f(x_N) \leq f^*,
\]
for any $N \geq N(\delta)$, where
\begin{equation}
  \label{eq:TheoreticalNumberOfIterations}
  N(\delta)
  \DefinedEqual
  \frac{16 n}{[(2 - \delta) \delta]^2}
  \leq
  \frac{16 n}{\delta^2}
\end{equation}
(see \cref{\LocalName{eq:PreliminaryNumberOfIterations}}).
  \section{Numerical Experiments}
\label{sec:NumericalExperiments}

Let us present preliminary computational results for
our algorithms as applied for solving the spectral linear regression
problem~\eqref{SpectralLinearRegression::eq:Problem}
using the setup from \cref{sec:SpectralLinearRegression}.
We set the target relative accuracy to one percent:
\begin{equation}
  \label{eq:TargetAccuracy}
  \delta \DefinedEqual 0.01,
\end{equation}
which is a typical choice in most engineering applications.

To be able to assess the performance of our optimization methods, we generate
data for problem~\eqref{SpectralLinearRegression::eq:Problem} in a special way.
Specifically, we choose the matrix~$C \in \RealMatrices{n}{m}$ to be
diagonal such that its largest element (in absolute value) is fixed and
is located in the top left corner:
\begin{gather}
  \label{eq:RequirementOnTargetMatrix}
  C = \DiagonalMatrix(1, c_2, \ldots, c_n),
  \qquad
  \Abs{c_i} \leq 1, \quad 2 \leq i \leq n,
\end{gather}
while the matrices $A_1, \ldots, A_d \in \RealMatrices{n}{m}$ are
constructed in such a way so that each of them has zero in the top left corner:
\begin{equation}
  \label{eq:RequiremenetOnBaseMatrices}
  A_i\UpperIndex{1, 1} = 0, \quad 1 \leq i \leq d.
\end{equation}
This way of generating data ensures that the optimal value for our
problem is known (see \cref{th:OptimalValueInExperiments}):
\[
  f^* = 1.
\]
The other diagonal elements $c_2, \ldots, c_n$ of~$C$ and all nonzero elements
of matrices $A_1, \ldots, A_d$ are generated randomly from the
standard uniform distribution on the interval~$\ClosedClosedInterval{-1}{1}$.

The data for our experiments in generated to be \emph{sparse}.
Specifically, each of the matrices~$A_1, \dots, A_d$ contains
only $s \DefinedEqual 5$ nonzero elements in each column.
The $s$ row indices of nonzero elements in each column $1 \leq j \leq m$
are randomly selected (without repetition) from the uniform distribution
on the set $\Set{1, \dots, n}$ if $j > 1$ and $\Set{2, \dots, n}$ if $j = 1$
(so that constraint~\eqref{eq:RequiremenetOnBaseMatrices} is respected).

The specific values of parameters~$d$, $n$ and~$m$, that we consider in the
experiments, are shown in \cref{tab:ParametersForExperiments},
together with the corresponding theoretical number of iterations~$N(\delta)$
that was computed according to \cref{eq:TheoreticalNumberOfIterations}.

\begin{table}
  \centering
  \begin{tabular}[t]{cccc}
    \toprule
    $d$ & $n$ & $m$ & $N(\delta)$ \\
    \midrule
    400 & 100 & 200 & 4\,040\,303 \\
    800 & 200 & 400 & 8\,080\,605 \\
    2\,000 & 500 & 1\,000 & 20\,201\,511 \\
    4\,000 & 1\,000 & 2\,000 & 40\,403\,021 \\
    \bottomrule
  \end{tabular}
  \caption{Problem instances used in our experiments.}
  \label{tab:ParametersForExperiments}
\end{table}

In what follows, we present the results in form of convergence plots for our
methods.
Each curve on such a plot displays the dependence of the relative
accuracy~$\delta_k \in \OpenOpenInterval{0}{1}$ of the current
approximate solution~$x_k$ constructed by the method against
the current iteration number~$k$ (or the total computational time taken by the
method up to iteration~$k$).
The accuracy~$\delta_k$ is defined as the smallest number such that
$(1 - \delta_k) f(x_k) \leq f^*$, i.e.,
\begin{equation}
  \label{eq:RelativeAccuracy}
  \delta_k = 1 - f(x_k) / f^*.
\end{equation}
Note from \cref{SpectralLinearRegression::eq:Problem} that
we cannot compute~$f(x_k)$ exactly as it requires computing the largest
singular value of the (potentially big) matrix~$X_k \DefinedEqual A x_k - C$.
Therefore, in practice, we actually approximate it by running the standard
Power Method for a sufficiently large number of iterations
(until the eigenvalue approximation stabilizes) to compute the
largest eigenvalue of the matrix~$X_k X_k\Transpose$ and then take the square
root.
Such an approximation is quite efficient and is sufficiently accurate for any
practical purposes.

The code for our experiments is written in C++ and uses the Eigen~3
library~\cite{EigenLibrary} for matrix computations.
It is compiled and run on a laptop with the Intel Core i7-8650U CPU,
16~GiB RAM, and Ubuntu~22.04 OS using the Clang~14 compiler.
For performing linear algebra operations, the Eigen library is allowed to use
all $8$ available threads.
\subsection{Gradient Method vs Dual Averaging}
\label{sec:ComparisonBetweenMethods}

\begin{figure}
  \newcommand{\PlotsPath}{NumericalExperiments/ComparisonBetweenMethods/plots}
  \centering

  \begin{subfigure}[b]{0.3\linewidth}
    \centering
    \includegraphics[width=\textwidth]{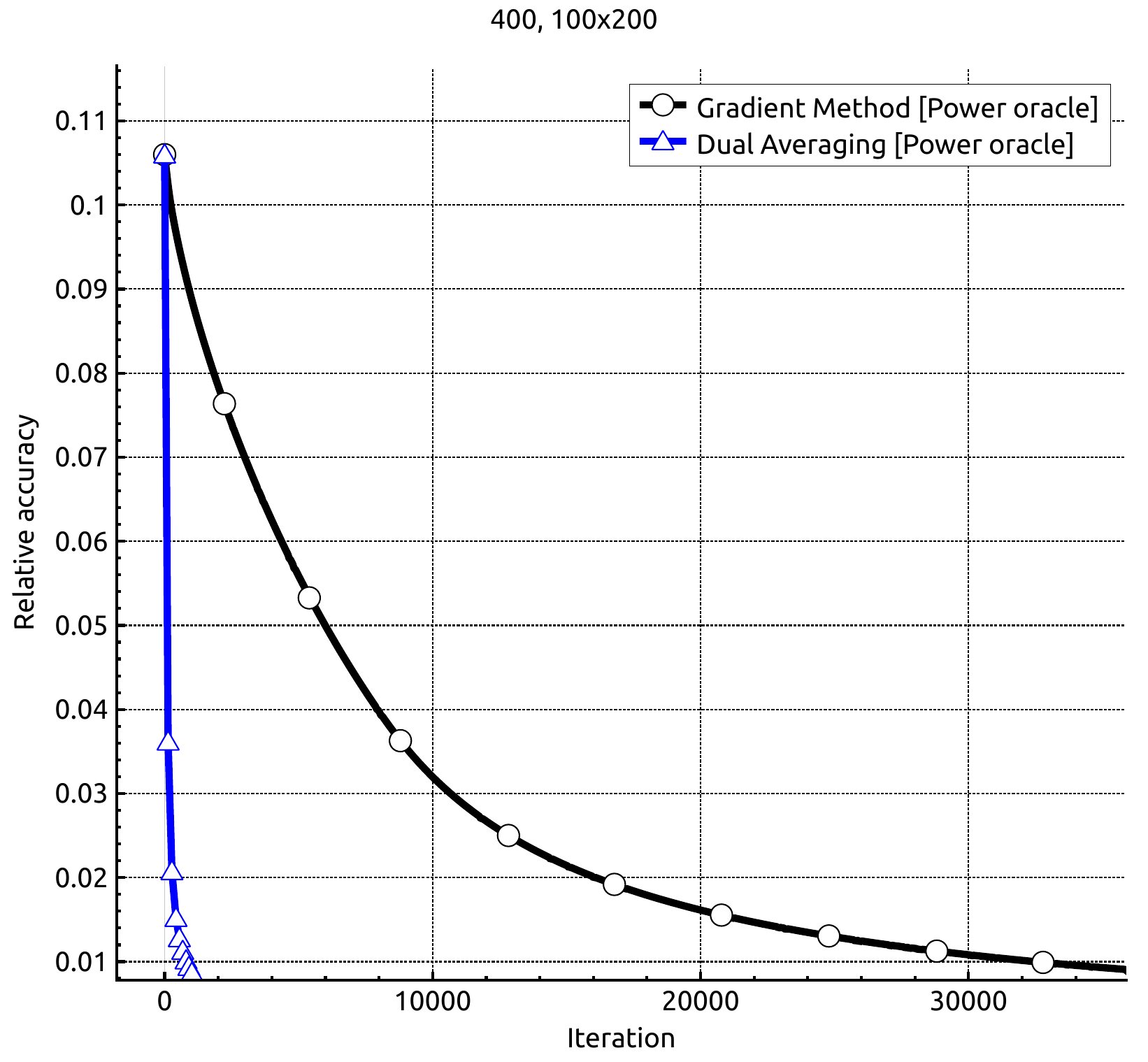}
  \end{subfigure}
  \begin{subfigure}[b]{0.3\linewidth}
    \centering
    \includegraphics[width=\textwidth]{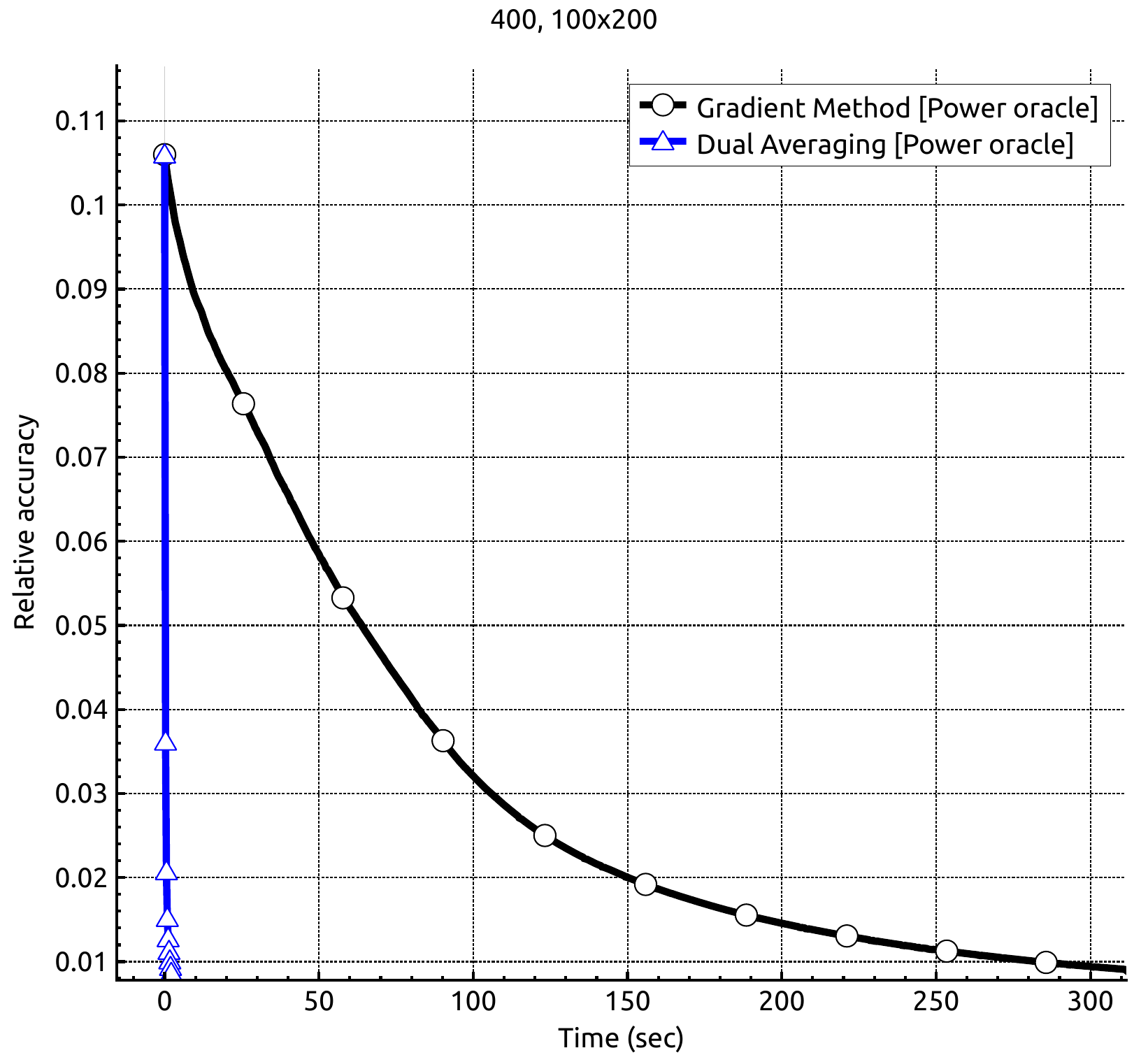}
  \end{subfigure}

  \begin{subfigure}[b]{0.3\linewidth}
    \centering
    \includegraphics[width=\textwidth]{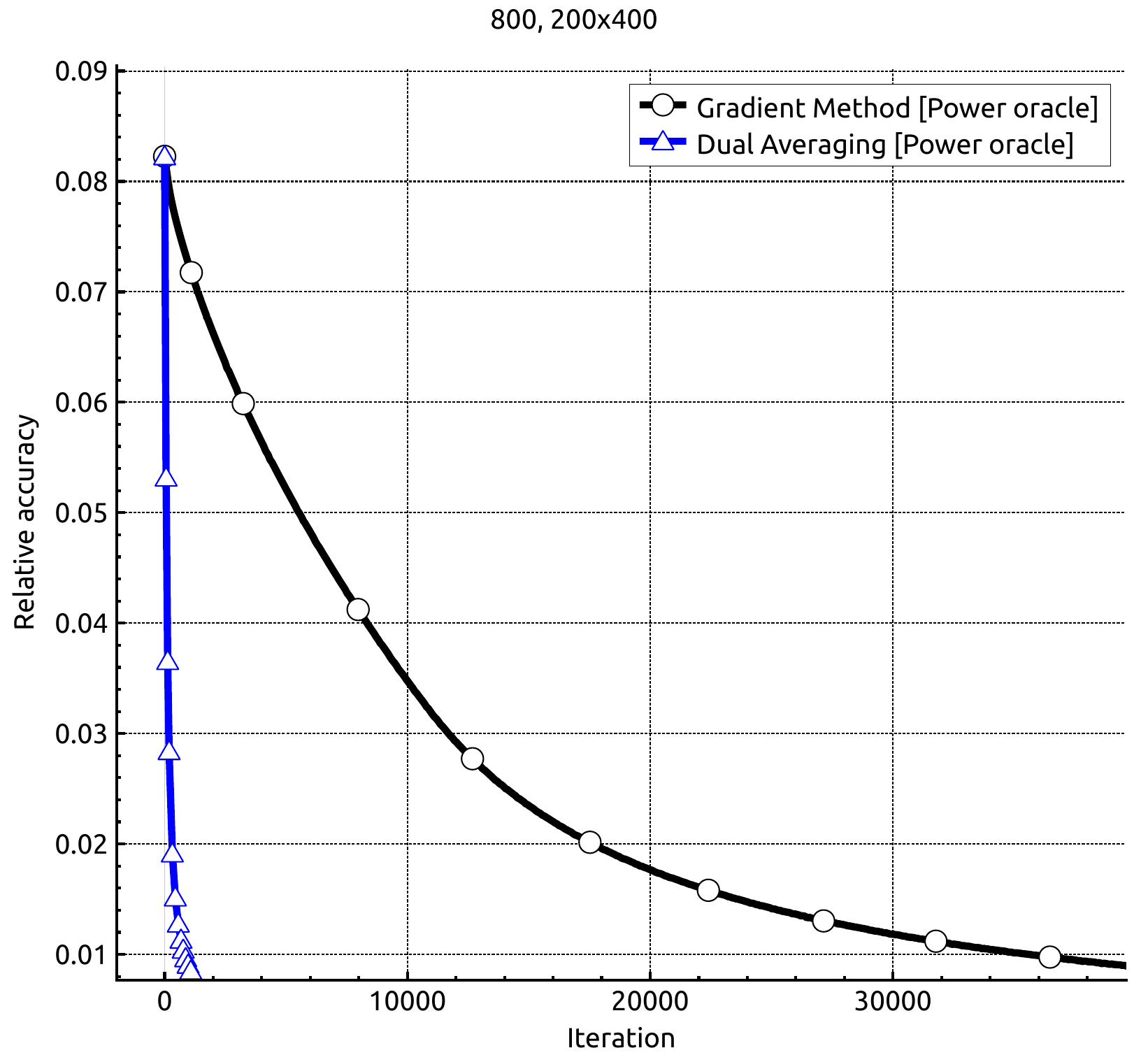}
  \end{subfigure}
  \begin{subfigure}[b]{0.3\linewidth}
    \centering
    \includegraphics[width=\textwidth]{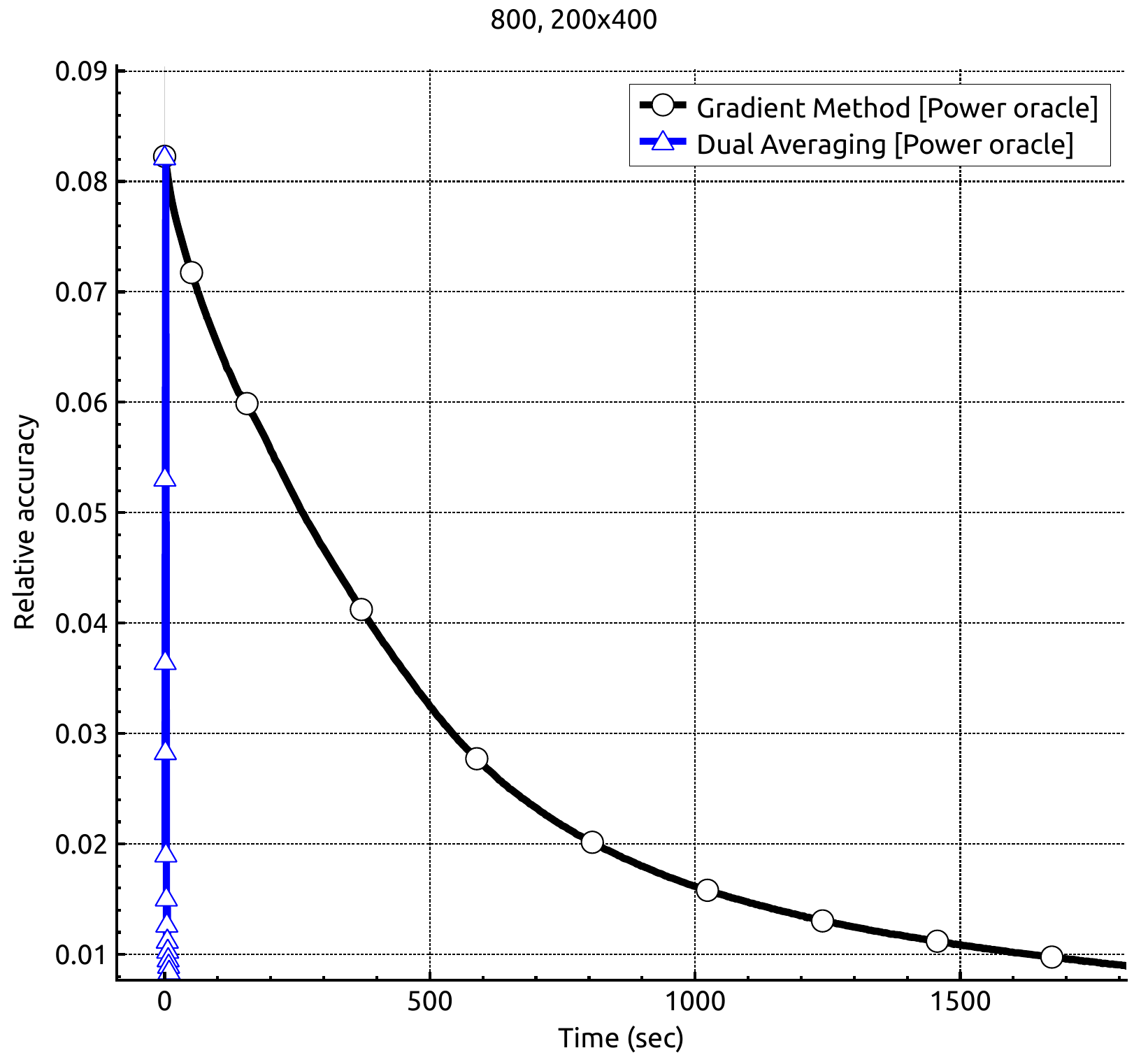}
  \end{subfigure}

  \caption{%
    Comparison between different methods for the first two problem instances
    from \cref{tab:ParametersForExperiments}.
  }
  \label{fig:ComparisonBetweenMethods-Results}
\end{figure}

In the first experiment, we compare the Gradient Method
(\cref{GradientMethodWithRelativelyInexactStochasticOracle::alg:Algorithm}
with constant step sizes~\eqref{SpectralLinearRegression::eq:StepSizes})
against the Dual Averaging method (\cref{alg:DualAveraging} with parameters
given by \cref{%
  eq:ChoiceOfOracleInaccuracy,%
  eq:ChoiceOfScalingCoefficients,%
  eq:ChoiceOfProxCoefficients%
}).
Both methods use the same oracle based on the Power algorithm for computing the
maximal eigenvector.

The results are shown in \cref{fig:ComparisonBetweenMethods-Results}
for two different instances of our problem.
The left column displays the convergence in terms of iteration numbers,
while the right columns displays the convergence in terms of the running time.

As we can see, there is a huge difference between the two methods:
the Gradient Method with fixed step sizes is significantly
slower than the Dual Averaging method with dynamically chosen parameters,
both in terms of the iteration number and, especially, the running time
(where the difference reaches \emph{several orders of magnitude!}).

Nevertheless, both methods have successfully reached the required target
accuracy~\eqref{eq:TargetAccuracy} in all cases.
What is very interesting, however, is that the actual number of iterations
it took for them to do that was much smaller than was predicted by the
worst-case theoretical estimate~$N(\delta)$ (see
\cref{tab:ParametersForExperiments}): by \emph{two orders of magnitude}
for the Gradient Method, and by approximately \emph{three to four orders of
magnitude} for Dual Averaging.
It is, of course, an interesting open question for future research---to
investigate why there is such a huge difference between theory and practice.
\subsection{Power Oracle vs Lanczos Oracle}

As we have seen in \cref{sec:ComparisonBetweenMethods}, the Gradient Method
with fixed step sizes is completely impractical.
Let us therefore consider only the Dual Averaging method now but look at the
difference between two different oracles: the Power oracle and the Lanczos
oracle.

The results are shown in \cref{fig:ComparisonBetweenOracles-Results},
where we now consider much larger problem instances than before.
As in \cref{fig:ComparisonBetweenMethods-Results}, the left column displays
the convergence in terms of iteration numbers, and the right column---in terms
of the running time.

As expected, the method with the Lanczos oracle is faster than the other one:
the gap between the two grows with the iteration counter (or the achieved
relative accuracy level) and, in the end, reaches, in our particular case,
approximately $4$ times, both in terms of the iteration number and the running
time.
Although the difference is not as dramatic as between the two different
methods from \cref{sec:ComparisonBetweenMethods}, it is still quite significant,
especially for large-scale problems that require hundreds and thousands
of seconds of computations.

  \newpage
  \appendix

  \section{Auxiliary Results}

\begin{lemma}
  \label{th:AuxiliaryClosednessResult}
  Let $Q \subseteq \VectorSpace{E}$ be a set,
  and let $\Map{A}{\VectorSpace{E}}{\VectorSpace{E}_1}$
  and $\Map{C}{\VectorSpace{E}_1}{\VectorSpace{E}_2}$ be linear transformations.
  Then, the following implication%
  \footnote{
    Hereinafter,
    $A(Q) \DefinedEqual \SetBuilder{A x}{x \in Q}$
    is the image of the set~$Q$ under the linear transformation~$A$.
  }
  holds:
  \[
    A(Q) + \ker C \ \text{is closed}
    \implies
    Q + \ker(C A) \ \text{is closed}.
  \]
\end{lemma}

\begin{proof}
  Let $A(Q) + \ker C$ be closed,
  and let $(z_k)_{k = 1}^\infty$ be a sequence in~$Q + \ker(C A)$
  converging to a point~$z \in \VectorSpace{E}$.
  Let us prove that $z \in Q + \ker(C A)$.
  Note that, for any $k \geq 1$, we have
  $A z_k \in A(Q) + A \ker(C A) \in A(Q) + \ker C$.
  Since $A$ is a continuous mapping
  (as a linear transformation between finite-dimensional vector spaces)
  and $z_k \to z$, it holds that $A z_k \to A z$.
  Furthermore, $A z \in A(Q) + \ker C$ since $A(Q) + \ker C$ is a closed set.
  Thus, $A z = A x + h$ for some $x \in Q$ and $h \in \ker C$.
  Consequently, $C A (z - x) = C h = 0$, which means that $z - x \in \ker(C A)$.
  But then $z = x + (z - x) \in Q + \ker(C A)$.
\end{proof}

\begin{lemma}
  \label{th:CriterionForClosednessOfProjection}
  Let $Q \subseteq \VectorSpace{E}$ be a set,
  $\VectorSpace{L} \subseteq \VectorSpace{E}$ be a linear subspace,
  $\VectorSpace{L}^c \subseteq \VectorSpace{E}$ be a complementary subspace
  to~$\VectorSpace{L}$, and let
  $\Map{P_{\VectorSpace{L}^c}}{\VectorSpace{E}}{\VectorSpace{L}^c}$
  be the projector%
  \footnote{%
    Specifically, if $x = x_{\VectorSpace{L}} + x_{\VectorSpace{L}^c}$ is the
    unique decomposition of~$x \in \VectorSpace{E}$ into the sum of elements
    from~$\VectorSpace{L}$ and~$\VectorSpace{L}^c$, respectively, then
    $P_{\VectorSpace{L}^c} x \DefinedEqual x_{\VectorSpace{L}^c}$.
  }
  of~$\VectorSpace{E}$ onto~$\VectorSpace{L}^c$
  corresponding to the decomposition
  $\VectorSpace{E} = \VectorSpace{L} \DirectSum \VectorSpace{L}^c$.
  Then,
  \[
    \text{$Q + \VectorSpace{L}$ is closed}
    \iff
    \text{$P_{\VectorSpace{L}^c}(Q)$ is closed}.
  \]
\end{lemma}

\begin{proof}
  Suppose $Q + \VectorSpace{L}$ is closed.
  Let $(u_k)_{k = 1}^\infty$ be an arbitrary sequence
  in~$P_{\VectorSpace{L}^c}(Q)$ converging to a point $u \in \VectorSpace{E}$.
  Let us prove that $u \in P_{\VectorSpace{L}^c}(Q)$.
  Clearly, $u \in \VectorSpace{L}^c$
  since $P_{\VectorSpace{L}^c}(Q) \subseteq \VectorSpace{L}^c$
  and~$\VectorSpace{L}^c$ is a closed set (as a linear subspace).
  On the other hand, since $u_k \in P_{\VectorSpace{L}^c}(Q)$
  for all $k \geq 1$, there exists a sequence~$(x_k)_{k = 1}^\infty$ in~$Q$
  such that $u_k = P_{\VectorSpace{L}^c} x_k$ for all $k \geq 1$.
  Then, $u_k = x_k - P_{\VectorSpace{L}} x_k \in Q + \VectorSpace{L}$
  for all $k \geq 1$.
  Since $Q + \VectorSpace{L}$ is a closed set and~$u_k \to u$,
  we have $u \in Q + \VectorSpace{L}$, i.e., $u = x - h$
  for some $x \in Q$ and $h \in \VectorSpace{L}$.
  Combining this with the fact that $u \in \VectorSpace{L}^c$, we conclude that
  $
    u
    =
    P_{\VectorSpace{L}^c} u
    =
    P_{\VectorSpace{L}^c} x
    \in
    P_{\VectorSpace{L}^c}(Q)
  $.
  This proves the ``$\Rightarrow$'' implication.

  The ``$\Leftarrow$'' implication follows from
  \cref{th:AuxiliaryClosednessResult}
  applied to $A \DefinedEqual P_{\VectorSpace{L}^c}$
  and $C \DefinedEqual I_{\VectorSpace{E}}$
  (the identity operator in~$\VectorSpace{E}$)
  as $\ker A = \VectorSpace{L}$ and $\ker C = \Set{0}$.
\end{proof}

\begin{lemma}
  \label{th:ConstancySpace}
  Let $\VectorSpace{L} \subseteq \VectorSpace{E}$ be a linear subspace,
  and let $\Map{f}{\VectorSpace{E}}{\RealField}$ be a convex function
  such that
  \[
    \Subdifferential f(x) \cap \VectorSpace{L}\OrthogonalComplement
    \neq
    \emptyset,
    \qquad
    \forall x \in \VectorSpace{E}.
  \]
  Then, $f$ is constant along~$\VectorSpace{L}$:
  \[
    f(x + h) = f(x),
    \qquad
    \forall x \in \VectorSpace{E}, \
    \forall h \in \VectorSpace{L}.
  \]
\end{lemma}

\begin{proof}
  Let $x \in \VectorSpace{E}$ and $h \in \VectorSpace{L}$.
  By our assumption, there is
  $f'(x) \in \Subdifferential f(x) \cap \VectorSpace{L}\OrthogonalComplement$.
  Hence,
  \[
    f(x + h) \geq f(x) + \DualPairing{f'(x)}{h} = f(x).
  \]
  Similarly, there exists
  $
    f'(x + h)
    \in
    \Subdifferential f(x + h) \cap \VectorSpace{L}\OrthogonalComplement
  $,
  and hence
  \[
    f(x) \geq f(x + h) + \DualPairing{f'(x + h)}{h} = f(x + h).
  \]
  Thus, $f(x + h) = f(x)$.
\end{proof}

\begin{lemma}
  \label{th:ExistenceOfMinimizerOnSubspace}
  \UsingNamespace{ExistenceOfMinimizerOnSubspace}
  Let $\Map{f}{\VectorSpace{E}}{\RealField}$ be a function,
  $Q \subseteq \VectorSpace{E}$ be a nonempty set,
  $\VectorSpace{L} \subseteq \VectorSpace{E}$ be a linear subspace,
  $\VectorSpace{L}^c$ be a complementary subspace to~$\VectorSpace{L}$,
  and let $\Map{P_{\VectorSpace{L}^c}}{\VectorSpace{E}}{\VectorSpace{L}^c}$
  be the projector of~$\VectorSpace{E}$ onto~$\VectorSpace{L}^c$
  corresponding to the decomposition
  $\VectorSpace{E} = \VectorSpace{L} \DirectSum \VectorSpace{L}^c$.
  Suppose that:
  \begin{EnumerateClaims}
    \item
      \LocalLabel{itm:FunctionIsConstantAlongSubspace}
      $f$ is constant along~$L$, i.e.,
      $f(x + h) = f(x)$
      for all $x \in \VectorSpace{E}$ and all $h \in \VectorSpace{L}$.
    \item
      \LocalLabel{itm:SumIsClosed}
      $Q + \VectorSpace{L}$ is a closed set.
    \item
      \LocalLabel{itm:FunctionIsClosed}
      $f$ is a closed function.
    \item
      \LocalLabel{itm:BoundedSublevelSets}
      $f$ restricted to~$P_{\VectorSpace{L}^c}(Q)$ has bounded sublevel sets%
      \footnote{
        This means that, for any $\alpha \in \RealField$, the set
        $\SetBuilder{u \in P_{\VectorSpace{L}^c}(Q)}{f(u) \leq \alpha}$
        is bounded.
      }.
  \end{EnumerateClaims}
  Then, $f$ has a minimizer on~$Q$.
\end{lemma}

\begin{proof}
  \UsingMasterNamespace{ExistenceOfMinimizerOnSubspace}
  In view of
  assumption~\LocalClaimReference{\MasterName{itm:FunctionIsConstantAlongSubspace}},
  we can reduce the problem of minimizing~$f$ on~$Q$ to that of minimizing~$f$
  on~$P_{\VectorSpace{L}^c}(Q)$:
  \begin{align*}
    \inf_{x \in Q} f(x)
    &=
    \inf_{u \in \VectorSpace{L}^c, h \in \VectorSpace{L}}
    \SetBuilder{f(u + h)}{u + h \in Q}
    =
    \inf_{u \in \VectorSpace{L}^c, h \in \VectorSpace{L}}
    \SetBuilder{f(u)}{u + h \in Q}
    \\
    &=
    \inf_{u \in \VectorSpace{L}^c}
    \SetBuilder{f(u)}{u + h \in Q \ \text{for some $h \in \VectorSpace{L}$}}
    =
    \inf_{u \in P_{\VectorSpace{L}^c}(Q)} f(u).
  \end{align*}
  In particular, if $f$ has a minimizer~$u^*$ on $P_{\VectorSpace{L}^c}(Q)$,
  then $f$ also has a minimizer on~$Q$, which is given by any $x^* \in Q$
  such that $P_{\VectorSpace{L}^c} x^* = u^*$
  (at least one such~$x^*$ exists by the definition
  of~$P_{\VectorSpace{L}^c}(Q)$).

  It remains to prove that $f$ has a minimizer on~$P_{\VectorSpace{L}^c}(Q)$.
  According to assumption~\LocalClaimReference{\MasterName{itm:SumIsClosed}}
  and \cref{th:CriterionForClosednessOfProjection},
  the set $P_{\VectorSpace{L}^c}(Q)$ is closed.
  Moreover, it is nonempty since $Q$ is assumed to be nonempty.
  Let $u_0 \in P_{\VectorSpace{L}^c}(Q)$ be an arbitrary point.
  It suffices to show that $f$ has a minimizer on the set
  $
    L_0
    \DefinedEqual
    \SetBuilder{u \in P_{\VectorSpace{L}^c}(Q)}{f(u) \leq f(u_0)}.
  $
  Clearly, $L_0 \neq \emptyset$ (it contains~$u_0$).
  Furthermore, $L_0$ is bounded
  (by assumption~\LocalClaimReference{\MasterName{itm:BoundedSublevelSets}})
  and closed as the intersection of two closed sets:
  $P_{\VectorSpace{L}^c}(Q)$
  and $\SetBuilder{u \in \VectorSpace{E}}{f(u) \leq f(u_0)}$
  (whose closedness follows from
  assumption~\LocalClaimReference{\MasterName{itm:FunctionIsClosed}}).
  Thus, $L_0$ is a nonempty compact set and $f$ is a closed function.
  Hence, by the Weierstrass extreme value theorem, there indeed exists a
  minimizer of~$f$ on~$L_0$.
\end{proof}

\begin{lemma}
  \label{th:OptimalValueInExperiments}
  Problem~\eqref{SpectralLinearRegression::eq:Problem} with data
  satisfying requirements~\eqref{eq:RequirementOnTargetMatrix}
  and~\eqref{eq:RequiremenetOnBaseMatrices} has an optimal solution
  $x^* = 0$ and the following optimal value:
  \[
    f^* = f(0) = \SchattenNorm{C}{\infty} = 1.
  \]
\end{lemma}

\begin{proof}
  It suffices to show that $f$ has a zero subgradient at~$x^* = 0$.
  Note that, for each $x \in \RealField^d$, we have $f(x) = F(A x - C)$, where
  $\Map{F}{\RealMatrices{n}{m}}{\RealField}$ is the spectral norm function
  $
    F(X)
    =
    \SchattenNorm{X}{\infty}
    =
    \max_{u \in \UnitSphere{n - 1}, v \in \UnitSphere{m - 1}}
    \InnerProduct{X v}{u}
  $
  and $\Map{A}{\RealField^d}{\RealMatrices{n}{m}}$
  is the linear operator defined in
  \cref{SpectralLinearRegression::eq:LinearOperator}.
  By standard calculus rules for subgradients, we know that,
  for any $F'(-C) \in \Subdifferential F(-C)$, we have
  $A\Adjoint F'(-C) \in \Subdifferential f(0)$,
  and, for any $X \in \RealMatrices{n}{m}$, we have
  $F'(X) \DefinedEqual u(X) [v(X)]\Transpose \in \Subdifferential F(X)$,
  where $u(X) \in \UnitSphere{n - 1}$ and $v(X) \in \UnitSphere{m - 1}$
  are such that $\InnerProduct{X v(X)}{u(X)} = F(X)$.
  According to \cref{eq:RequirementOnTargetMatrix}, we can take
  $u(-C) \DefinedEqual -e_{1, n}$ and $v(-C) = e_{1, m}$, where
  $e_{1, n} \DefinedEqual (1, 0, \ldots, 0) \in \RealField^n$
  and $e_{1, m} \DefinedEqual (1, 0, \ldots, 0) \in \RealField^m$.
  This gives us $F'(-C) = -e_{1, n} e_{1, m}\Transpose$.
  Consequently, $f'(0) = -A\Adjoint (e_{1, n} e_{1, m}\Transpose)$
  is the vector with elements
  $
    [f'(0)]\UpperIndex{i}
    =
    -\DualPairing{A_i}{e_{1, n} e_{1, m}\Transpose}
    =
    -A_i\UpperIndex{1, 1}
    =
    0
  $
  (see \cref{eq:RequiremenetOnBaseMatrices})
  for any $1 \leq i \leq d$.
\end{proof}

  \printbibliography

\end{document}